\documentclass[10pt,letterpaper,reqno]{amsart}

\title[Semi-infinite programming using polynomial interpolants]{Semi-infinite programming using high-degree polynomial interpolants and semidefinite programming}
\author{D{\'a}vid Papp}
\address{North Carolina State University, Department of Mathematics. E-mail: \url{dpapp@ncsu.edu}.}
\date{\today. \revisiontwo{The author thanks the referees and Sercan Y{\i}ld{\i}z for their numerous constructive suggestions on how to improve the manuscript. This material was based upon work partially supported by the National Science Foundation under Grant DMS-1127914 to the Statistical and Applied Mathematical Sciences Institute. Any opinions, findings, and conclusions or recommendations expressed in this material are those of the author(s) and do not necessarily reflect the views of the National Science Foundation.}}

\usepackage[margin=1.5in]{geometry}
\usepackage{amsmath}
\usepackage{amssymb}
\usepackage{enumerate}
\usepackage[colorlinks=true, citecolor=black, linkcolor=black, urlcolor=black, setpagesize=false, plainpages=false, pdfpagelabels]{hyperref}
\usepackage{doi}
\usepackage{graphicx}
\usepackage{epstopdf}
\usepackage{booktabs}

\newcommand{\deletethis}[1]{{}}

\deletethis{
\newcommand{\vA}{\mathbf{A}}
\newcommand{\vb}{\mathbf{b}}
\newcommand{\vB}{\mathbf{B}}

\newcommand{\vD}{\mathbf{D}}

\newcommand{\vf}{\mathbf{f}}
\newcommand{\vg}{\mathbf{g}}

\newcommand{\vI}{\mathbf{I}}

\newcommand{\vM}{\mathbf{M}}

\newcommand{\vs}{\mathbf{s}}
\newcommand{\vt}{\mathbf{t}}
\newcommand{\vu}{\mathbf{u}}

\newcommand{\vW}{\mathbf{W}}
\newcommand{\vx}{\mathbf{x}}
\newcommand{\vy}{\mathbf{y}}
\newcommand{\vX}{\mathbf{X}}

\newcommand{\vz}{\mathbf{z}}

\newcommand{\vzero}{\mathbf{0}}

}

\newcommand{\vA}{{A}}
\newcommand{\vb}{{b}}
\newcommand{\vB}{{B}}

\newcommand{\vD}{{D}}

\newcommand{\vf}{{f}}
\newcommand{\vg}{{g}}

\newcommand{\vI}{{I}}

\newcommand{\vM}{{M}}

\newcommand{\vs}{{s}}
\newcommand{\vt}{{t}}
\newcommand{\vu}{{u}}

\newcommand{\vW}{{W}}
\newcommand{\vx}{{x}}
\newcommand{\vy}{{y}}
\newcommand{\vX}{{X}}

\newcommand{\vz}{{z}}

\newcommand{\vzero}{{0}}

\newcommand{\real}{\mathbb{R}}

\newcommand{\cI}{\mathcal{I}}

\newcommand{\Sk}{\mathbb{S}^k}
\newcommand{\Ski}{\mathbb{S}^{k_i}}
\newcommand{\Sm}{\mathbb{S}^m}
\newcommand{\Smp}{\Sm_+}
\newcommand{\T}{\mathrm{T}}
\newcommand{\Mm}{\mathcal{M}}

\newcommand{\lcm}{\operatorname{lcm}}
\newcommand{\den}{\operatorname{den}}

\newcommand{\conv}{\operatorname{conv}}
\DeclareMathOperator*{\dom}{dom}
\newcommand{\tr}{\operatorname{tr}}
\newcommand{\defeq}{\ensuremath{\overset{\mathrm{def}}{=}}}

\theoremstyle{plain}
\newtheorem{definition}{Definition}
\newtheorem{lemma}[definition]{Lemma}
\newtheorem{proposition}[definition]{Proposition}
\newtheorem{theorem}[definition]{Theorem}

\theoremstyle{definition}
\newtheorem{example}{Example}

\newcommand{\revision}[1]{#1}
\newcommand{\revisiontwo}[1]{#1}
\newcommand{\revisionthree}[1]{{#1}}

\begin{document}
	
\begin{abstract}
	In a common formulation of semi-infinite programs, the infinite constraint set is a requirement that a function parametrized by the decision variables is nonnegative over an interval. If this function is sufficiently closely approximable by a polynomial or a rational function, then the semi-infinite program can be reformulated as an equivalent semidefinite program, \revision{which in turn can be solved with interior-point methods very efficiently to high accuracy}. On the other hand, solving this semidefinite program is challenging if the polynomials involved are of high degree, due to numerical difficulties and bad scaling arising both from the polynomial approximations and from the fact that the semidefinite programming constraints coming from the sum-of-squares representation of nonnegative polynomials are badly scaled. We combine \revision{polynomial} function approximation techniques and polynomial programming to overcome these numerical difficulties, using sum-of-squares interpolants. Specifically, it is shown that the conditioning of the reformulations using sum-of-squares interpolants does not deteriorate with increasing degrees, and problems involving sum-of-squares interpolants of hundreds of degrees can be handled without difficulty. The proposed reformulations are sufficiently well scaled that they can be solved easily with every commonly used semidefinite programming solver, such as SeDuMi, SDPT3, and CSDP. Motivating applications include convex optimization problems with semi-infinite constraints and semidefinite conic inequalities, such as those arising in the optimal design of experiments. Numerical results align with the theoretical predictions; in the problems considered, available memory was the only factor limiting the degrees of polynomials, to approximately 1000.
\end{abstract}
	
	\maketitle
	
\section{Introduction}\label{sec:introduction}


A linearly constrained semi-infinite convex optimization problem with infinitely many linear constraints \revisiontwo{indexed by a one-dimensional index set} can be posed as:
\begin{equation}
\begin{split}
\textrm{minimize}_\vx &\quad f(\vx)\\
\textrm{subject to}   &\quad \vA(\vt)\vx \leq \vb\quad \forall\,\revisiontwo{\vt\in T= [0,1]}\\
&\quad \vx \in X
\end{split} \label{eq:SILP}
\end{equation}
with respect to the decision variables $\vx$, \revisiontwo{where the set $X\subseteq\real^n$ is convex, closed and bounded, $f$ is convex and continuous on $X$, $\vb$ is an $m$-dimensional vector, $\vA$ is a $[0,1]\to\real^m$ function}. \revisiontwo{Slightly more generally, the index set $T$ may be a finite union of closed intervals}. Even without any restrictions on the dependence of $\vA$ on $\vt$, \eqref{eq:SILP} is a convex optimization problem, and the Weierstrass extreme value theorem guarantees that its minimum is attained.

In many applications, $\vx$ represents a function $p\colon T\mapsto\real$ that is known to belong to a given finite dimensional linear space (that is, $\vx$ is the coefficient vector of $p$ in some fixed basis), and the infinite constraint set represents $p(t)\geq 0$ for all $t\in T$. Similar constraints on the derivatives of a differentiable function, such as $\tfrac{dp(t)}{dt}\geq 0$ (implying monotonicity) or $\tfrac{d^2p(t)}{dt^2}\geq 0$ (implying convexity) can also be represented in a similar fashion. Optimization models incorporating such constraints have been used, for example, in arrival rate estimation \cite{PappAlizadeh-2013}, and in semi-parametric density estimation with and without shape constraints \cite{PappAlizadeh-2014}.

In most applications, including all the ones mentioned above, the index set $T$ is low-dimensional, and is often simply an interval or a two-dimensional rectangular box \cite{LopezStill-2007,GobernaLopez-98,ReemtsenRuckmann-98}. \revisiontwo{In this paper, we focus on the one-dimensional case.}

\deletethis{
Most research in semi-infinite optimization theory and algorithms focus on linearly constrained problems, appealing to the fact that the general convex problem \eqref{eq:SICP} is equivalent to the linearly constrained problem (SILP) formulated below:
\begin{equation}
\begin{split}
\textrm{minimize}_\vx&\quad f(\vx)\\
\textrm{subject to}  &\quad \vu_i^\T \vx - g_{i,\vt}^*(\vu_i)\leq 0\quad \forall\,\vt\in T\text{ and } \vu_i\in \dom g_{i,\vt}^*\\
&\quad \vx \in X,
\end{split}\tag{SILP}\label{eq:SILP}
\end{equation}
where $g_{i,\vt}^*$ denotes the conjugate function of the $i$th component of $\vg(\cdot,\vt)$.
}

Several algorithms have been proposed to solve semi-infinite linear and semi-infinite convex programming problems, including cutting plane and cutting surface methods such as \cite{KortanekNo1993, Betro2004, MehrotraPapp-2014}, local reduction methods \cite{GobernaLopez-98}, exchange methods \cite{ZhangWuLopez-2010}, and homotopy methods \cite{Liu-2007}. See also \cite{LopezStill-2007} for a relatively recent review on semi-infinite convex programming, including an overview on numerical methods with plenty of references.

\revisiontwo{The primary motivation behind this work is the case when solutions need to be computed with high accuracy, which is often the case when the results of the optimization are inputs of further numerical computations that are sensitive to errors in their inputs. Such problems arise for example in the computation of optimal designs of experiments, where the points in $T$ for which the semi-infinite constraints are active (the zeros of the optimal $b-A(\cdot)x$) need to be determined.} We shall revisit this problem, and discuss it in more detail, in Section \ref{sec:design-linear-regression}. \revision{There are several other areas where certifying the nonnegativity of high-degree univariate polynomials and (more generally, the positive semidefiniteness of univariate polynomial matrices) is of high importance, these include numeric-symbolic computation in algebraic geometry,  optimal control, and signal processing; see, e.g., \cite{MeniniTornambe2015, AylwardItaniParrilo2007}.}

The methods mentioned above for solving \eqref{eq:SILP} rarely have faster than linear convergence, which is sufficient for computing  low-accuracy solutions, but makes them particularly challenging to apply when the solutions are needed to be computed with high accuracy. \revisiontwo{The convergence rate of discretization methods is particularly well understood \cite{Still2001}.} On the other hand, in special cases that admit ``nice'' finitely constrained convex optimization formulations, or a linear cone optimization formulation over ``nice'' cones such as symmetric cones, superlinearly convergent (or even polynomial time) interior point methods can be utilized to compute solutions with high accuracy. These special cases include when all the sets and functions involved in the problem formulation are \emph{semidefinite representable} (see below), which implies that \eqref{eq:SILP} can be posed as a semidefinite optimization problem.

Let $\Sk$ denote the set of $k\times k$ real symmetric matrices, and $\Sk_+$ denote the set of positive semidefinite matrices from $\Sk$. We shall also use the notation $\vM\succcurlyeq\vzero$ to denote that the matrix $\vM$ is positive semidefinite. Recall \cite{VandenbergheBoyd-1996}, \cite[Sec.~4.2]{BenTalNemirovski-2001} that a set $S\subseteq\real^n$ is \emph{semidefinite representable} if for some $k\geq 1$ and $\ell \geq 0$ there exist affine functions $A:\real^n\mapsto\Sk$ and $C:\real^\ell\mapsto\Sk$ such that the set $S$ can be characterized by a linear matrix inequality in the following way:
	\[ S = \{\vs \in \real^n\;|\; \exists\, \vu \in \real^\ell \colon A(\vs) + C(\vu) \succcurlyeq \vzero \}. \]
With this definition, we can also say that a continuous \emph{function} is \emph{semidefinite representable} if all of its \revision{(closed)} lower level sets are semidefinite representable. (For maximization problems, a concave objective function is semidefinite representable if its closed upper level sets are semidefinite representable.)

For the rest of the paper we assume that
\begin{enumerate}
	\item the objective function $f$ is semidefinite representable;
	\item the set $X$ is semidefinite representable;
	\item \revisiontwo{the set $T$ is the union of finitely many closed and bounded intervals}; and
	\item the elements of $\vA(t)$ are continuous on $T$.
\end{enumerate}

These assumptions are practically nonrestrictive, and the first \revisiontwo{two} of them imply that aside from the infinite constraint set $\vA(t)\vx \leq \vb$, the problem \eqref{eq:SILP} can be posed as a semidefinite program. If, additionally, $\vA(t)$ belongs to a ``nice'' set of univariate functions, for which the set $\left\{\vx\in\real^n\,\middle|\,\vA(t)\vx \leq \vb\; \forall\,t\in [a,b]\right\}$ is also semidefinite representable for every given interval $[a,b]$, then \eqref{eq:SILP} can be reformulated as a semidefinite program, and can in principle be solved to high accuracy with existing semidefinite programming solvers.

\revisiontwo{The assumptions (3) and (4) imply} that the elements of $\vA(t)$ can be approximated arbitrarily closely in the uniform norm by polynomials \cite[Chapter 1]{Timan-1963}. Since the set of nonnegative polynomials is also semidefinite representable \cite{KarlinStudden-1966,BenTalNemirovski-2001}, this suggests the following approach:
\begin{enumerate}
	\item Replace all entries of $\vA(t)$ by polynomials that approximate the entries within a chosen $\varepsilon>0$ in the uniform norm, and
	\item reformulate the resulting approximate optimization problem as a semidefinite program, and solve this semidefinite program using interior point methods.
\end{enumerate}
While this approach is fairly ``obvious'', it can easily be viewed as entirely impractical, not the least because the semidefinite representation of nonnegative polynomials (when the polynomials are represented in the standard basis) involves Hankel matrices that are notoriously ill-conditioned \cite{Tyrtyshnikov-1994,Beckermann-1997}, and existing interior point methods are not designed to be able to handle such problems numerically. Although a change of basis to some orthogonal basis, such as the Chebyshev or Legendre polynomial basis, can somewhat mitigate this problem, the resulting semidefinite programs can still be difficult, if not impossible, to solve with existing semidefinite programming solvers, once the degree of the polynomials involved exceeds about 40 \cite{Papp-2011}. In contrast, a very close approximation of $\vA(t)$ may require polynomials with at least hundreds of degrees.

\revision{We shall mention that in Step (1) above one may wish to use rational function approximations in place of polynomials. Everything in this paper generalizes to rational function approximation, but for ease of presentation we focus on polynomial approximations only. See Section 7 (Discussion) for a brief description on how to incorporate rational function approximations.}

\revisionthree{As the main contribution of this paper, we demonstrate that using a suitable problem-dependent basis of polynomial interpolants to represent the approximators of $\vA(t)$, the above approach of polynomial approximation and semidefinite reformulation can be carried out efficiently, completely circumventing the aforementioned numerical difficulties and scaling issues, even if the polynomials involved are of high degree. The approach uses polynomial interpolants. The use of interpolants in polynomial optimization problems was first proposed in \cite{LofbergParrilo-2004}, where it was shown that it is numerically preferable to represent a polynomial nonnegativity constraint in terms polynomial interpolants instead of the standard monomial basis when optimizing a univariate polynomial using semidefinite programming. We show that a combination of constructive approximation techniques from \cite{Higham-2004,Trefethen-2013} and the semidefinite programming approach of \cite{LofbergParrilo-2004} yields a very efficient interior-point approach to a large class of semi-infinite programming problems.}

\revisionthree{From an implementation perspective, this means combining efficient and numerically stable algorithms for manipulating polynomial interpolants and standard semidefinite programming solvers. Beyond its simplicity, an additional advantage of this approach over developing tailor-made algorithms to handle the polynomial nonnegativity constraints is that this approach is directly applicable to convex optimization problems that involve semidefinite constraints other than the ones arising from polynomial nonnegativity. A family of such problems is presented in Section \ref{sec:design-linear-regression}}.

The rest of the paper is structured as follows. In Section \ref{sec:best-approximations}, we briefly review existing results regarding best polynomial approximations of smooth functions, and their representations as Lagrange or Hermite interpolants. In Section \ref{sec:sos-interpolants}, we extend the results of \cite{LofbergParrilo-2004} and \cite{deKlerkElabwabiDenHertog-2006} to provide semidefinite representations of sum-of-squares interpolants. Following that, we discuss how sum-of-squares Lagrange interpolants admit well-scaled semidefinite representations that can be handled with any existing semidefinite programming solver, in \mbox{Section \ref{sec:scaling}}. Section \ref{sec:upsampling} is concerned with a formally minor, but numerically critical issue of representing low-degree polynomials as high-degree interpolants in a fashion that does not affect the numerical properties of the semidefinite representation. This is an essential ingredient of optimization problems involving sum-of-squares polynomials of different degrees. Numerical examples are presented in Section \ref{sec:numerical}. The first couple of examples illustrate the individual components developed throughout Sections \ref{sec:best-approximations}--\ref{sec:upsampling}, while the statistical applications (concerning optimal designs of experiments) demonstrate how the developed methodology can be used to solve convex semi-infinite programs with both functional and conic constraints using high-degree polynomial approximations and semidefinite programming. Section \ref{sec:conclusions} concludes the paper with a discussion on the advantages and limitations of the approach.

\section{Two- and one-sided approximations}\label{sec:best-approximations}
In order to keep the paper self contained, we shall very briefly review a few key results from constructive approximation theory that we shall make use of. The reader will find much more detail on these ideas in the references provided with the Propositions below. We shall connect these results to conic programming representations of best and near-best polynomial approximations of smooth functions in the next section.

\revision{Consider a differentiable function $f\colon[-1,1]\mapsto\real$ that we wish approximate by a polynomial. If the approximating polynomial is allowed to take both larger and smaller values than $f$, we speak of a \emph{two-sided} approximation, while \emph{one-sided} approximation refers to the case when the approximating polynomial is required to stay above (or below) $f$. Best and near-best approximations can be particularly easily \revisiontwo{constructed} using Lagrange interpolation (for two-sided approximation) and Hermite interpolation (for one-sided approximation).}

\subsection{Lagrange interpolants and near-optimal two-sided approximation}
Consider a differentiable function $f\colon[-1,1]\mapsto\real$ that we wish to uniformly approximate by a polynomial, and take a set of distinct points $t_0,\dots,t_n$ in the interval $[-1,1]$. Then there is a unique polynomial $p_n$ of degree $n$ satisfying $p(t_j) = f(t_j)$ for each $j=0,\dots,n$. We call $p_n$ the \emph{polynomial (Lagrange) interpolant} of $f$ on the \emph{interpolation points} $t_0,\dots,t_n$. If the interpolation points are chosen well, increasing the number of interpolation points decreases the \emph{interpolation error} $\|f-p_n\|\revisiontwo{_\infty} \defeq \max_{-1\leq t\leq 1} |f(t)-p_n(t)|$. A particularly useful choice of interpolation points is the set of \emph{Chebyshev points of the second kind}, defined by the formula
\begin{equation}\label{eq:ChebyshevPoints-2ndkind}
t_\ell = \cos(\ell\pi/n)\quad \ell=0,\dots,n.
\end{equation}
Another frequently used interpolation points are the \emph{Chebyhev points of the first kind}; they are defined by the formula
\begin{equation}\label{eq:ChebyshevPoints-1stkind}
t_\ell = \cos((\ell + 1/2)\pi/(n+1))\quad \ell=0,\dots,n.
\end{equation}
Many other schemes exist for choosing interpolation points.

Using the Chebyshev points defined in \eqref{eq:ChebyshevPoints-2ndkind}, the interpolation error converges to zero at a rate dependent on an appropriate measure of smoothness of $f$:


\begin{proposition}[\protect{\cite[Theorem 7.2]{Trefethen-2013}}]\label{thm:prop1}
	For an integer $k\geq 1$, let $f$ and its derivatives through $f^{(k-1)}$ be absolutely continuous on $[-1,1]$, and suppose that the $k$th derivative $f^{(k)}$ is of bounded variation $V_{f^{(k)}}$. Then for every $n\geq k+1$, the interpolant $p_n$ through the Chebyshev points \eqref{eq:ChebyshevPoints-2ndkind} satisfies
	\[\|f-p_n\|\revisiontwo{_\infty} \leq \frac{4V_{f^{(k)}}}{\pi k(n-k)^k}.\]
\end{proposition}

The convergence is even faster, has a geometric rate, for analytic functions; we shall omit the somewhat complicated details of the constructive formulation of this theorem.

Interpolants through Chebyshev points are also nearly as good approximations as any other polynomial approximant of the same degree can be:
\begin{proposition}[\protect{\cite[Theorem 16.1]{Trefethen-2013}}]\label{thm:prop2}
Let $f$ be a continuous function on $[-1,1]$, and $p_n$ its degree $n$ interpolant through the Chebyshev points \eqref{eq:ChebyshevPoints-2ndkind}. Then for every polynomial $q$ of degree $n$, we have
	\[\|f-p_n\|\revisiontwo{_\infty} \leq \left(2+\tfrac{2}{\pi} \log(n+1)\right)\|f-q\|\revisiontwo{_\infty}.\]
\end{proposition}
Specifically, as $\left(2+\tfrac{2}{\pi} \log(n+1)\right)< 10$ for all $n\leq 200000$, this proposition says that in every application considered in this paper, interpolants on Chebyshev points lose at most one digit of accuracy compared to the best polynomial approximant of the same degree.

\subsection{Hermite interpolants and optimal one-sided approximation}

In some applications it is imperative that the computed approximate optimal solution $\vx^*$ to \eqref{eq:SILP} be feasible (and not just approximately feasible).
However, because Lagrange interpolants oscillate around $f$, the polynomial approximation of  $\vA(t)\vx^*$ might also oscillate around the true function, and as a result, might be slightly infeasible to \eqref{eq:SILP}. If this is undesirable, and $\vx$ is componentwise nonnegative, then this problem can be avoided by using \emph{one-sided approximations}, that is, polynomial approximants that are always below (or above) the true function. An elementary approach is to replace the approximant $p_n$ of Proposition \ref{thm:prop1} by $(p_n-\varepsilon)$, or $(p_n+\varepsilon)$ for upper approximations, where $\varepsilon = 4V_{f^{(k)}}/(\pi k(n-k)^k)$. With the help of Proposition \ref{thm:prop2} one can argue that this is also a near-best one-sided approximation in the uniform norm.

If the derivatives of $f$ are of constant sign, we can also use the best one-sided approximation of $f$ in the $L_1$-norm $\|f\|_1 \defeq \int_{-1}^{1}|f(t)|dt$, which can be characterized as an (Hermite) interpolant at the zeros of certain orthogonal polynomials. These zeros depend only on the degree of the approximation, but not the approximated function $f$ itself. The following proposition summarizes the characterization of these polynomial approximants. (See, for example, \cite[Section 1.4]{DunklXu-2001} for a definition of the Legendre and Jacobi polynomials referred to in the next Proposition.)

\begin{proposition}[\cite{BojanicDeVore-1966}]
	\label{thm:BojanicDeVore}
Assume that $f$ is a continuous function on $[-1,1]$ whose $(n+1)$-st derivative $f^{(n+1)}$ is nonnegative on $(-1,1)$.
\begin{enumerate}
\item If $n=2k-1$ is \emph{odd}, let $t_1,\dots,t_k$ be the zeros of the Legendre polynomial of degree $k$. Then the degree-$n$ polynomial of best approximation of $f$ \emph{from below} in the $L_1$ norm is the unique polynomial $p_n$ satisfying
\[ p_n(t_\ell) = f(t_\ell)\;\text{ and }\; p_n'(t_\ell) = f'(t_\ell), \quad \ell=1,\dots,k. \]
\item If $n=2k$ is \emph{even}, let $t_1,\dots,t_k$ be the zeros of the Jacobi polynomial $P_k^{(0,1)}$. Then the degree-$n$ polynomial of best approximation of $f$ \emph{from below} in the $L_1$ norm is the unique polynomial $p_n$ satisfying
\[ p_n(-1) = f(-1) \quad\text{and}\quad p_n(t_\ell) = f(t_\ell)\;\text{ and }\; p_n'(t_\ell) = f'(t_\ell), \quad \ell=1,\dots,k. \]
\end{enumerate}
\end{proposition}
\noindent Analogous characterizations of best approximations \emph{from above} are also given in \cite{BojanicDeVore-1966}.

\section{Sum-of-squares interpolants}\label{sec:sos-interpolants}
We say that a polynomial is \emph{sum-of-squares} if it can be written as a (finite) sum of squared polynomials. Specifically, we write $p\in SOS_{2k}$ if $p$ is a polynomial (of degree at most $2k$) that can be written as a sum of squares of polynomials of degree $k$. It is well-known \cite{PolyaSzego-1976b} that a univariate polynomial $p$ of degree $2k$ is nonnegative on the entire real line if and only if $p\in SOS_{2k}$. Similarly, a polynomial $p$ of degree $n$ is nonnegative over $[-1,1]$ if and only if it can be written as a weighted sum of squared polynomials \cite{Lukacs-1918}, either in the form of
\begin{equation}
p(t)=(1+t)q(t)+(1-t)r(t), \quad q\in SOS_{2k-2},\; s\in SOS_{2k-2} \qquad \text{if }n=2k-1,\label{eq:wsos-odd}
\end{equation}
or in the form
\begin{equation}
p(t)=(1+t)(1-t)q(t)+s(t), \quad q\in SOS_{2k-2},\; s\in SOS_{2k}, \qquad \text{if }n=2k.\phantom{-1 }\label{eq:wsos-even}
\end{equation}

This, in turn, yields a semidefinite representation of the set of nonnegative polynomials of a fixed degree, using the fact that the cone of sums of squares of functions from any finite dimensional functional space is semidefinite representable \cite{Nesterov-2000,Parrilo-2003}. \revisiontwo{The precise form of this semidefinite representation depends on the bases that the polynomials being squared ($q$, $r$, and $s$) and the sum-of-squares polynomial ($p$) are represented in.}

For the purposes of this paper, we need a representation that uses only the values of $p$ and its derivatives at prescribed interpolation points, so that the polynomial approximations of the functions involved in \eqref{eq:SILP} need not be explicitly constructed, but one may work directly with sampled values of the original functions to be approximated. For Lagrange interpolants at the points $t_0,\dots,t_n$, this is equivalent to representing the squared polynomials in an arbitrary basis, while representing $p$ in the Lagrange basis polynomials corresponding to the interpolation points $t_0,\dots,t_n$ \cite{LofbergParrilo-2004}. \revision{The theorem below shows that for every fixed set of interpolation points, the coefficient vectors of nonnegative polynomials in the interpolant basis are a linear image of the cone of positive semidefinite matrices.} We use the notation $A\bullet B \defeq \sum_{i,j}A_{ij}B_{ij}$ to denote the component-wise (Frobenius) inner product.
\begin{theorem}\label{thm:Lagrange-SOS}
	Let $t_0,\dots,t_{2k}\in\real$ be distinct \revision{interpolation points} and $f_0,\dots,f_{2k}\in\real$ be arbitrary \revision{function values prescribed at these points}. Fix an arbitrary basis $p_0,\dots,p_k$ of polynomials of degree $k$. Then there is a \revision{nonnegative} polynomial $q\in SOS_{2k}$ satisfying $q(t_\ell)=f_\ell$ for each $\ell=0,\dots,2k$ if and only if there exists a $(k+1)\times(k+1)$ positive semidefinite matrix $X$ satisfying
	\begin{equation}\label{eq:Lagrange-SOS}
	A^{(\ell)} \bullet X = f_\ell \quad \ell=0,\dots,2k,\qquad \text{where}\quad A^{(\ell)}_{ij} = p_i(t_\ell) p_j(t_\ell).
	\end{equation}
\end{theorem}
We omit the proof, as this proposition is subsumed by Theorem \ref{thm:Hermite-SOS} below, which provides a similar characterization for Hermite interpolants. The following is a generalization of both Theorem \ref{thm:Lagrange-SOS} above and (the main, unnumbered, results of) Sections 3.2 and 3.3 of \cite{deKlerketal-2006}.

\begin{theorem}\label{thm:Hermite-SOS}
	Let $t_1, \dots, t_k \in \real$ be distinct \revision{interpolation points}, let $\revision{m_1},\dots,m_k$ and $d$ be nonnegative integers satisfying $2d+1=\sum_{\ell=1}^{k} (m_\ell+1)$, and let $f_\ell^{(m)}\in\real$ be arbitrary prescribed values of the $m$th derivative at $t_\ell$ for every $\ell=1,\dots,k$ and $m=0,\dots,m_\ell$. Also fix an arbitrary basis $p_0,\dots,p_d$ of polynomials of degree $d$. Then there is some \revision{nonnegative} polynomial $q\in SOS_{2d}$ satisfying
\[q^{(m)}(t_\ell) = f_\ell^{(m)} \qquad \ell=1,\dots,k,\quad m=0,\dots,m_\ell\]
if and only if there exists a \revisiontwo{$(d+1)\times (d+1)$} positive semidefinite matrix $X$ satisfying
\begin{equation}\label{eq:Hermite-SOS}
A^{(\ell,m)} \bullet X = f_\ell^{(m)} \quad \ell=0,\dots,2k,\qquad \text{where}\quad A^{(\ell,m)}_{ij} = \tfrac{\mathrm{d}^m}{\mathrm{d}t^m}(p_i(t) p_j(t))\big\rvert_{t=t_\ell}.
\end{equation}
\end{theorem}
\begin{proof}
Using the shorthand $p(t)$ to denote the column vector $(p_0(t), \dots, p_d(t))^\T$,
$q\in SOS_{2d}$ if and only if there exists some \revisiontwo{$(d+1)\times (d+1)$} positive semidefinite matrix $X$ with which
$q(t) = p(t)^\T X p(t) = (p(t)p(t)^\T) \bullet X$ for every $t\in\real$. Differentiating both sides of this equation, we obtain
\begin{equation}\label{eq:dmdq} \tfrac{\mathrm{d}^m}{\mathrm{d}t^m}q(t) = \tfrac{\mathrm{d}^m}{\mathrm{d}t^m}(p(t)p(t)^\T) \bullet X \qquad \forall\,t,
\end{equation}
where the differentiation of the matrix on the right-hand side is understood componentwise.

Since the prescribed derivative values and the degree determine $q$ uniquely, equation \eqref{eq:dmdq} holds for every $m\geq 0$ and every $t\in\real$ if and only if it holds for each $t_\ell$, $\ell=0,\dots,k$ with $m$ from $0$ up to $m_\ell$, which is precisely the system of equations \eqref{eq:Hermite-SOS} in our claim.
\end{proof}

\section{Scaling}\label{sec:scaling}
One of the key difficulties in working with polynomials of high degree is that numerical difficulties arise if the polynomials involved are represented in an unsuitable basis, such as the monomial basis, as it is customary in the sum-of-squares literature. In the representation of Theorem \ref{thm:Lagrange-SOS} one can freely choose both the interpolation points and the basis $p$. The choice of Chebyshev points of the first kind and appropriately scaled Chebyshev polynomials works particularly well. Recall that the \emph{Chebyshev polynomials of the first kind} are the sequence of polynomials of increasing degree defined by the recursion
\begin{equation}\label{eq:ChebyshevPolynomials-1stkind}
T_0(t) = 1, \quad T_1(t) = t, \quad T_i(t) = 2t T_{i-1}(t) - T_{i-2}(t) \quad i=2,3,\dots
\end{equation}

\deletethis{
\begin{theorem}
	Let $t_0,\dots,t_n$ be the Chebyshev points defined in \eqref{eq:ChebyshevPoints-1stkind}, and let $T_i$ be the Chebyshev polynomials given in \eqref{eq:ChebyshevPolynomials-1stkind}. Define the polynomials $p_0,\dots,p_n$ by
	\[ p_0 = (n+1)^{-1/2} T_0; \qquad p_i = ((n+1)/2)^{-1/2} T_i\quad i=1,\dots,n.\]
	Then
	\[ \sum_{\ell=0}^n p_i(t_\ell)p_j(t_\ell) = \begin{cases}1 & \text{if } i=j\\ 0& \text{if } i\neq j\\\end{cases}.\]
\end{theorem}
\begin{proof}
\end{proof}
}

\deletethis{
\begin{proposition}[\protect{\cite[eq.~(3.30)]{GilSeguraTemme-2007}}]
	Let $t_0,\dots,t_n$ be the Chebyshev points defined in \eqref{eq:ChebyshevPoints-1stkind}, and let $T_i$ be the Chebyshev polynomials given in \eqref{eq:ChebyshevPolynomials-1stkind}. Then
	\[ \sum_{\ell=0}^n T_i(t_\ell)T_j(t_\ell) = K_i\delta_{ij},\]
	where $K_0 = n+1$, $K_i = (n+1)/2$ when $i\geq 1$, and $\delta$ is the Kronecker symbol	$\delta_{ij}=\begin{cases}1 & \text{if } i=j\\ 0& \text{if } i\neq j\\\end{cases}$.
\end{proposition}
}

The following lemma states that if we represent the polynomials to be squared in the appropriately scaled Chebyshev basis, and use Chebyshev points as the interpolation points, then the representation \eqref{eq:Lagrange-SOS} is perfectly scaled.

\begin{lemma}\label{lem:orthogonality}
	Let $t_0,\dots,t_{2k}$ be the Chebyshev points given in \eqref{eq:ChebyshevPoints-1stkind} (with $n=2k$), and define $p_0 = \sqrt{\tfrac{1}{2k+1}} T_0$ and $p_i = \sqrt{\tfrac{2}{2k+1}}T_i$ for $i=1,\dots,k$, where the $T_i$ are the Chebyshev polynomials given in \eqref{eq:ChebyshevPolynomials-1stkind}. Then the vectors $(p_i(t_0), \dots, p_i(t_{2k}))$ for $i=0,\dots,k$ form an orthonormal system.
\end{lemma}
\begin{proof}
The statement is an easy consequence of the \emph{discrete orthogonality relation} of Chebyshev polynomials \cite[eq.~(3.30)]{GilSeguraTemme-2007}: for every $n>0$ and $0\leq i,j\leq n$,
\[ \sum_{\ell=0}^n T_i(t_\ell)T_j(t_\ell) = K_i\delta_{ij},\]
where $K_0 = n+1$, $K_i = (n+1)/2$ when $i\geq 1$, and $\delta$ is the Kronecker symbol. 
Applying this identity, we have the following:
\begin{enumerate}
	\item If $i=j=0$, then
	\[ \sum_{\ell=0}^{2k} p^2_0(t_\ell) = \frac{1}{2k+1}\sum_{\ell=0}^{2k}T_0^2(t_\ell) = 1. \]
	\item If $i>j=0$, then
\[ \sum_{\ell=0}^{2k} p_i(t_\ell)p_j(t_\ell) = \frac{\sqrt{2}}{2k+1}\sum_{\ell=0}^{2k}T_i(t_\ell)T_0(t_\ell) = 0. \]
	\item If $i\geq 1$ and $j \geq 1$, then
	\[ \sum_{\ell=0}^{2k} p_i(t_\ell)p_j(t_\ell) = \frac{2}{2k+1}\sum_{\ell=0}^{2k}T_i(t_\ell)T_j(t_\ell) = \delta_{ij}. \qedhere \]
\end{enumerate}
\end{proof}

 
We can also express this relation in terms of the dual constraints. If $y_\ell$ denotes the dual variable corresponding to the linear equation $A^{(\ell)} \bullet X = f_\ell$ in \eqref{eq:Lagrange-SOS}, then dual constraint corresponding to the primal variable $X$ is that the matrix
\begin{equation}
 Y(y) \defeq \sum_{\ell=0}^{2k} y_\ell p(t_\ell)p(t_\ell)^\T = P^\T\operatorname{diag}(y) P, \label{eq:dualconstr}
\end{equation}
where $P = (p_i(t_\ell))_{\ell,i}$, is positive semidefinite. Choosing each $p_i$ and $t_\ell$ as suggested by Lemma \ref{lem:orthogonality} we obtain a matrix $P$ with orthonormal columns. Factoring, or computing eigenvalues of $Y(y)$ for different values of $y$ is therefore not numerically challenging even if the number of interpolation points (and the degree of the polynomials involved) is in the thousands. (See Section \ref{sec:numerical} for numerical examples.) An additional advantage is that algorithms that compute the values of Chebyshev polynomials at Chebyshev points to arbitrary accuracy are readily available \cite{Clenshaw-1955,Trefethen-2013,chebfun}; also note that these values $T_i(t_\ell)$, and therefore the coefficient matrices $A^{(\ell)}$ in \eqref{eq:Lagrange-SOS} need only be computed once, offline, for every value of $n$.

The case of general interpolation points is in principle similarly easy. For every set of points $t_0,\dots,t_{2k}$ one can find a basis $p_0,\dots,p_k$ of polynomials of degree $k$ satisfying the discrete orthogonality relation
\[ \sum_{\ell=0}^n p_i(t_\ell)p_j(t_\ell) = \delta_{ij},\]
by taking an arbitrary basis, and applying an orthogonalization procedure, e.g., QR factorization \cite[Sec.~19]{Higham-2002}. It is important to note that the representation \eqref{eq:Lagrange-SOS} does \emph{not} require that the basis $p_0,\dots,p_k$ is explicitly identified or expressed in any particular basis, only the values of the basis polynomials are needed at the interpolation points. Throughout the orthogonalization procedure one can work directly with the values of the basis polynomials at the prescribed points. For example, the initial basis can be the Chebyshev polynomial basis, as in that basis stable evaluation of the basis polynomials is easy \cite{Clenshaw-1955,chebfun}, and then the orthogonalization procedure applied to the vectors of function values directly computes the values of $p_i(t_\ell)$ for each interpolation point $t_\ell$ for the orthogonalized basis $p$.

The same procedure is applicable to the weighted-sum-of-squares representations \eqref{eq:wsos-odd} and \eqref{eq:wsos-even} of polynomials that are nonnegative over an interval. In the dual constraint \eqref{eq:dualconstr}, the $(\ell,i)$-th entry $p_i(t_\ell)$ of the coefficient matrix $P$ is replaced by $w(t_\ell)^{1/2}p_i(t_\ell)$, where $w$ is the weight polynomial (in the case of polynomials over $[-1,1]$, this is the polynomial $1-t$, $1+t$, or $1-t^2$, depending on the parity of the degree), and it is this weighted coefficient matrix that needs to be orthogonalized for a perfectly scaled representation of the weighted-sum-of-squares constraint.

\section{High-degree representation of low-degree polynomials}\label{sec:upsampling}

If the polynomial approximations of \eqref{eq:SILP} involve interpolants of different degrees, it may be necessary to lift the lower degree interpolants into the space of higher degree ones. When polynomials are represented in the monomial basis or in an orthogonal basis, this is straightforward: the coefficients of the higher degree terms simply need to be set to zero. The analogous constraint for interpolants is that the low-degree polynomial must take consistent values at the interpolation points used to represent the high-degree polynomials. Since the evaluation of a polynomial at a given point is a linear functional, the mapping $B\colon\real^{n+1}\mapsto\real^{N+1}$ from the values of a degree-$n$ polynomial on a given set of $n+1$ interpolation points to the values on a larger set of $N+1$ interpolation points is linear, therefore this mapping can be represented by a system of linear equality constraints.

For example, suppose that the infinite constraint set can be written in the form
\[ p(t)-q(t) \geq 0\quad \forall\,t\in[-1,1],\]
where $q$ is a given degree-$N$ polynomial represented as an interpolant on the $N+1$ Chebyshev points of the second kind \eqref{eq:ChebyshevPoints-2ndkind}, whereas $p$ is an degree-$n$ polynomial to be optimized, with $n<N$, represented as an interpolant on $n+1$ Chebyshev points. To represent this constraint using sum-of-squares interpolants, $p-q$ needs to be a degree-$N$ sum-of-squares interpolant, and the variable $p$ must be ``upsampled'' and represented as a degree-$N$ interpolant.

The coefficient matrix of these constraints can be determined using interpolation formulae. It is important to note that although the previous section shows that the sum-of-squares representation of nonnegative interpolants can always be scaled, regardless of the location of the interpolation points, the problem of polynomial interpolation is inherently ill-conditioned in general, meaning that small changes in the values of the degree-$n$ polynomial can result in large changes in the upsampled values \cite{BerrutTrefethen-2004}. On the other hand, if the low-degree polynomial is an interpolant on Chebyshev points, or any other point set with asymptotic density $(1-x^2)^{-1/2}$, then the interpolation problem is well-conditioned; moreover, the coefficients of the upsampling constraints can be computed efficiently and in a numerically stable manner, using barycentric Lagrange interpolation \cite{Higham-2004}. In practice, these computations can be conveniently carried out using the Matlab package \texttt{chebfun}.

It is for this reason that in all our numerical examples in this paper, all polynomials are represented as interpolants using Chebyshev points as interpolation points.

\section{Applications and numerical experiments}\label{sec:numerical}

\revision{
The complete algorithm for the solution of semi-infinite programs given in the form \eqref{eq:SILP} can be summarized as follows:}
\begin{enumerate}[1.]
	\item \revision{Choose a convenient family of interpolation points. If the application does not prescribe them, use Chebyshev points of the second kind defined in \eqref{eq:ChebyshevPoints-2ndkind}.}

	\item \revision{Find a componentwise polynomial approximation $P(t)$ of each component of $A(t)$, expressed as an interpolant, by evaluating $A(\cdot)$ at each interpolation point. Use \mbox{Proposition \ref{thm:prop1}} to compute the number of points that ensure that the interpolants have small uniform approximation error. (In the examples in this paper we always use enough points to obtain a uniform error less than double machine precision.)}
	
	\item \revision{Reformulate the polynomial inequalities $b-P(t)x\geq 0$ as linear semidefinite optimization constraints using Theorem \ref{thm:Lagrange-SOS} (if Lagrange interpolation is used) or Theorem \ref{thm:Hermite-SOS} (for Hermite interpolation).}
	
	\item \revision{If the degree of the components of $P(\cdot)$ is high, use the procedure in Section \ref{sec:scaling} to orthogonalize the semidefinite representation of the polynomial constraints. If Chebyshev points were chosen in Step 1, Lemma \ref{lem:orthogonality} gives the orthonormal representation in closed form, and this step can be omitted.}
	
	\item \revision{Solve the resulting convex optimization problem with a suitable convex optimization method. If the convex constraints $x\in X$ and the objective function $f$ are semidefinite representable (as defined in Section \ref{sec:introduction}), we can use interior point methods for semidefinite programming.}
\end{enumerate}

\revision{Note that as long as the approximation $P(t)$ for $A(t)$ is sufficiently close, the original problem and the polynomial approximation are numerically equivalent. Specifically, infeasibility or unboundedness in \eqref{eq:SILP} is detected in the last step.}

\revision{In the rest of the section we illustrate the method using a number of examples.}

\subsection{Polynomial envelopes}

Our first example demonstrates that the computational infrastructure presented in Sections \ref{sec:sos-interpolants} and \ref{sec:scaling} is indeed capable of handling high-degree polynomials without any numerical difficulties. Consider the following problem: given degree-$d$ polynomials $p_1,\dots,p_m$, find the greatest degree-$n$ polynomial lower approximation of $\min(p_1,\dots,p_m)$, where the minimum is understood pointwise. Formally, we seek the optimal solution to
\begin{equation}
\begin{split}
\textrm{maximize}_p &\quad \int_{-1}^1 p(t) dt\\
\textrm{subject to}\, &\quad p(t)\leq p_i(t)\;\; \forall\,t\in [-1,1]\quad i=1,\dots,m.
\end{split}\label{eq:envelope}
\end{equation}
All polynomials involved can be represented as interpolants on the same $\max(n,d)+1$ points. The decision variables are the function values $p(t_\ell)$, $\ell=1,\dots,\max(n,d)+1$ at the interpolation points $t_\ell$. The nonnegativity of the polynomials $p_i(t)-p(t)$ can be formulated as these polynomials being weighted sums of squares. \revisiontwo{The integral in the objective can be replaced by the sum $\sum_{\ell} p(t_l) w_{\ell}$ with appropriately chosen weights $w_{\ell}$ for an explicit representation as a linear function of the decision variables}. For this example, we assume that $n\geq d$ so that we do not have to worry about the upsampling issue discussed in the previous section.

Random instances were generated by drawing uniformly random integer coefficients from $[-9,9]$ for each $p_i$ represented in the Chebyshev basis. \revisiontwo{For ease of implementation, the dual problem},
\begin{equation}
\begin{split}
\textrm{minimize}_\vy &\quad \sum_{i,\ell} p_i(t_\ell)y_{i\ell}\\
\textrm{subject to}   &\quad \sum_i y_{i\ell}=w_\ell\quad \forall\,\ell,\\
                      &\quad \sum_\ell (1+t_\ell)A^{(\ell)} y_{i\ell} \succcurlyeq 0\quad \forall\,i,\\
                      &\quad \sum_\ell (1-t_\ell)A^{(\ell)} y_{i\ell} \succcurlyeq 0\quad \forall\,i,
\end{split}\label{eq:envelope-dual}
\end{equation}
was solved, after orthogonalizing the matrix inequalities using the procedure outlined in Section \ref{sec:scaling}. We employed three different solvers, Sedumi \cite{sedumi}, SDPT3 version 4 \cite{sdpt3}, and CSDP version 6.2 \cite{csdp}, each running in Matlab 2014a, \revisiontwo{interfaced using the OPTI Toolbox \cite{opti} version 2.20}, to confirm that the semidefinite formulations can indeed be solved with off-the-shelf SDP solvers. A variety of values for the number of polynomials $m$ as well as the degrees $n$ and $d$  were tried. \revisiontwo{The fact that the $\vA^{(\ell)}$ matrices are of rank one can also in principle be exploited by dual solvers \cite{LofbergParrilo-2004}, although the solvers used in this study do not take advantage of this.}

In each case, the optimal polynomial interpolant $p$ was  recovered from the optimal dual solution:  $p(t_\ell)$ is the optimal dual variable corresponding to the linear equality constraint $\sum_i y_{i\ell}=w_\ell$ in \eqref{eq:envelope-dual}.

For solvers that do not handle linear equality constraints well (specifically, those that represent $Ax=b$ as a pair of inequalities $b\leq Ax\leq b$), it is useful to note that the polynomials $p_i$ can be assumed to satisfy $p_i(t_\ell)\leq 0$ for all $i$ and $\ell$, without any loss of generality, which in turn allows us to assume that the primal variables $p(t_\ell)$ in \eqref{eq:envelope} are constrained (redundantly) to be non-positive. With this modification, the equality constraint in the dual problem \eqref{eq:envelope-dual} becomes an inequality $\sum_i y_{i\ell}\leq w_\ell\; \forall\,\ell$.

\begin{example}\label{ex:poly-env}
Figure \ref{fig:polyenv_simple} depicts three quintic polynomials, along with three best polynomial lower approximations of their pointwise minimum. The degrees of the approximating polynomials are 5, 15, and 75, respectively. The 75-degree lower approximation is visually nearly indistinguishable from the minimum of the three polynomials. The computations for this plot were carried out using Sedumi.

\begin{figure}[htb]
	\includegraphics[width=0.5\columnwidth]{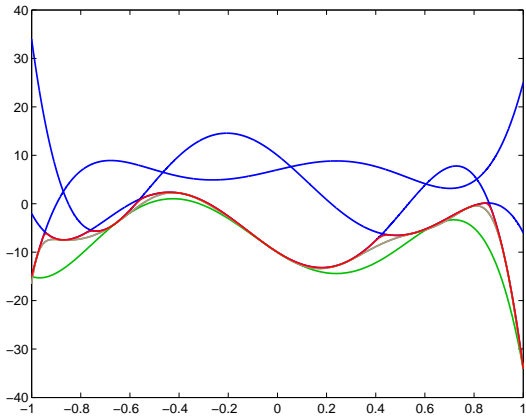}
	\caption{Three polynomials of degree 5, along with their best lower polynomial approximations of degree 5, 15, and 75.}
	\label{fig:polyenv_simple}
\end{figure}

To find the optimal polynomials with the highest numerically possible accuracy, we set the Sedumi accuracy goal \texttt{eps} to zero so that the solver keeps iterating as long as it can make any progress. Figure \ref{fig:polyenv_errors} shows the plot of the difference between $\min_i p_i$ and the three polynomial lower approximations. Only the sections of the plots close to the $x$-axis are shown, in order to demonstrate that the resulting optimal polynomials are computed to sufficiently high accuracy that the points of contact (the points where $\min_i p_i(t) = p(t)$) can be separated, and computed to several digits of accuracy.

\begin{figure}[htb]
	\includegraphics[width=0.45\columnwidth]{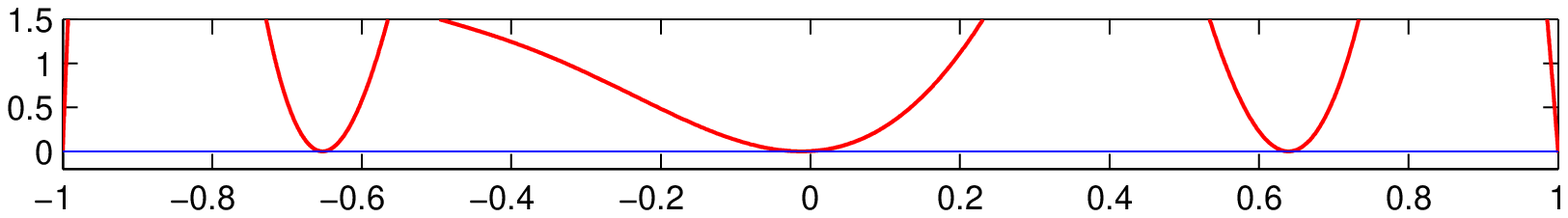} \includegraphics[width=0.45\columnwidth]{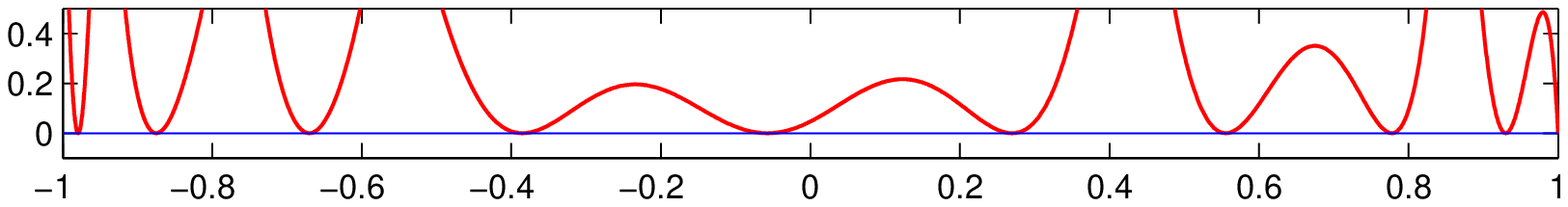}\\
	\includegraphics[width=0.8\columnwidth]{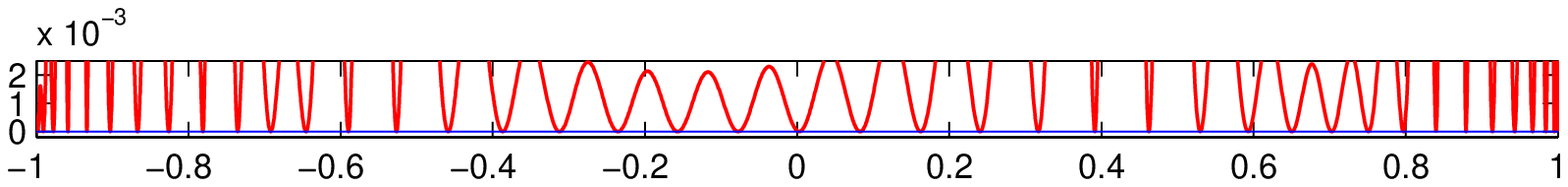}
	\caption{Pointwise difference between the three optimal polynomials and $\min_i p_i$ from Figure \ref{fig:polyenv_simple}. Only the near-zero section of the plots are shown; the number of contact points can be easily read off the diagrams. The polynomials have degrees 5, 15, and 75, respectively.}
	\label{fig:polyenv_errors}
\end{figure}
\end{example}

To test the limits of the approach when applied to polynomials of very high degree, similar problems were solved for higher values of $n$, with the three SDP solvers mentioned above (Sedumi, SDPT3, and CSDP). As before, to obtain the highest possible accuracy, we set tolerances and accuracy goals to zero so that the solvers keep iterating as long as they can make any progress. Otherwise, default parameter settings were used with each solver.

With each solver, as the sizes of the SDPs grow quadratically with the degree of the polynomials involved, the available memory became a bottleneck. Therefore, we reduced the number of constraints to $m=2$, and then increased $n$ as shown in the tables below. Using a standard desktop computer with 32GB RAM, the degree was increased until the solvers ran out of memory, and were unable to solve the problem. The number of nonzeros in the constraint matrix of the semidefinite program, along with the number of iterations, the solver running time, and the final duality gap for each run of Sedumi is shown in Table~\ref{tbl:ex1-highdegree-sedumi}; the same solver statistics (without repeating the problem statistics) for SDPT3 are shown in Table~\ref{tbl:ex1-highdegree-sdpt3}, and for CSDP in Table~\ref{tbl:ex1-highdegree-csdp}. It is apparent from the results that the solvers are able to solve even the largest instances, involving polynomials of degree 1000, without any numerical difficulty, and the memory constraint is the only bottleneck.

\begin{table}
	\centering
	\begin{tabular}{rcccccc}
		\toprule
		$n+1\!\!$ & \# of nonzeros & \# of iterations & solver time [s] & primal inf. & dual inf. & duality gap\\
		\midrule
		100   & 0.5 M    & 23    & \phantom{1221}5   & $5.0\cdot10^{-10}$ & $3.0\cdot10^{-14}$ & $3.25\cdot10^{-14}$\\
		200   & 4.0 M    & 21    & \phantom{121}44   & $3.5\cdot10^{-10}$ & $1.5\cdot10^{-13}$ & $9.04\cdot10^{-13}$\\
		300   & 13.5 M   & 24    & \phantom{12}215   & $1.7\cdot10^{-10}$ & $1.4\cdot10^{-14}$ & $6.23\cdot10^{-15}$\\
		400   & 32.1 M   & 21    & \phantom{12}547   & $2.3\cdot10^{-10}$ & $1.7\cdot10^{-13}$ & $6.01\cdot10^{-14}$\\
		500   & 62.6 M   & 19    & \phantom{1}1128   & $1.1\cdot10^{-9}$ & $9.1\cdot10^{-13}$ & $2.43\cdot10^{-13}$\\
		600   & 108 M    & 20    & \phantom{1}2456   & $2.6\cdot10^{-9}$ & $2.0\cdot10^{-12}$ & $4.56\cdot10^{-13}$\\
		700   & 171 M    & 21    & \phantom{1}4847   & $4.8\cdot10^{-10}$ & $3.0\cdot10^{-13}$ & $6.19\cdot10^{-14}$\\
		800   & 256 M    & 21    & \phantom{1}8670   & $7.2\cdot10^{-10}$ & $3.2\cdot10^{-13}$ & $5.76\cdot10^{-14}$\\
		900   & 321 M    & 20    & 12969  & $1.9\cdot10^{-9}$ & $1.1\cdot10^{-12}$ & $1.80\cdot10^{-13}$\\
		1000  & 501 M    & 19    & 19875  & $5\cdot10^{-9}$ & $2.8\cdot10^{-12}$ & $4.19\cdot10^{-13}$\\		
		\bottomrule\\
	\end{tabular}
	\caption{Solver statistics from Sedumi from the solution of Example 1. Instances of the optimization problem \eqref{eq:envelope}-\eqref{eq:envelope-dual} was solved for $m=2$, $d=5$, and different values of the degree $n$. (That is, $n+1$ in the heading is the number of interpolation points.) M in the second column stands for millions. The last three columns show the relative infeasibility of the optimal primal and dual solutions, and the relative duality gap. Larger problems ($n+1 \geq 1100)$ could not be solved because of memory constraints.}
	\label{tbl:ex1-highdegree-sedumi}
\end{table}

\begin{table}
	\centering
	\begin{tabular}{rccccc}
		\toprule
		$n+1\!\!$ & \# of iterations & solver time [s] & primal inf. & dual inf. & duality gap\\
		\midrule
		100   & 24  & \phantom{1069}4 & $6.3\cdot10^{-10}$ & $1.1\cdot10^{-11}$ & $5.2\cdot10^{-12}$\\
		200   & 25  & \phantom{106}25 & $1.4\cdot10^{-9}$ & $3.9\cdot10^{-12}$ & $5.7\cdot10^{-12}$\\
		300   & 29  & \phantom{10}107 & $8.5\cdot10^{-9}$ & $1.0\cdot10^{-12}$ & $1.5\cdot10^{-11}$\\
		400   & 26  & \phantom{10}264 & $2.7\cdot10^{-9}$ & $5.1\cdot10^{-12}$ & $1.7\cdot10^{-11}$\\
		500   & 29  & \phantom{10}695 & $3.4\cdot10^{-9}$ & $4.3\cdot10^{-13}$ & $1.6\cdot10^{-11}$\\
		600   & 30  & \phantom{1}1395 & $9.7\cdot10^{-10}$ & $1.6\cdot10^{-12}$ & $3.1\cdot10^{-10}$\\               
		700   & 30  & \phantom{1}2527 & $2.2\cdot10^{-9}$ & $9.5\cdot10^{-13}$ & $1.7\cdot10^{-10}$\\
		800   & 33  & \phantom{1}4732 & $3.0\cdot10^{-8}$ & $2.3\cdot10^{-13}$ & $5.4\cdot10^{-12}$\\
		900   & 30  & \phantom{1}6724 & $5.6\cdot10^{-10}$ & $4.2\cdot10^{-12}$ & $9.1\cdot10^{-10}$\\
		1000  & 31  &           10505 & $3.9\cdot10^{-10}$ & $2.2\cdot10^{-13}$ & $2.4\cdot10^{-11}$\\		
		\bottomrule\\
	\end{tabular}
	\caption{Solver statistics from SDPT3 from the solution of Example 1. Instances of the optimization problem \eqref{eq:envelope}-\eqref{eq:envelope-dual} was solved for $m=2$, $d=5$, and different values of the degree $n$. Larger problems ($n+1 \geq 1100)$ could not be solved because of memory constraints.}
	\label{tbl:ex1-highdegree-sdpt3}
\end{table}

\begin{table}
	\centering
	\begin{tabular}{rccccc}
		\toprule
		$n+1\!\!$ & \# of iterations & solver time [s] & primal inf. & dual inf. & duality gap\\
		\midrule
		100   & 17  & \phantom{1888}1 & $1.89\cdot10^{-11}$ & $6.68\cdot10^{-13}$ & $1.74\cdot10^{-9}$\\
		200   & 19  & \phantom{188}10 & $2.57\cdot10^{-12}$ & $1.63\cdot10^{-12}$ & $3.71\cdot10^{-10}$\\
		300   & 21  & \phantom{188}45 & $9.13\cdot10^{-12}$ & $2.29\cdot10^{-10}$ & $1.70\cdot10^{-9}$\\
		400   & 19  & \phantom{18}136 & $1.83\cdot10^{-11}$ & $7.02\cdot10^{-12}$ & $6.38\cdot10^{-9}$\\
		500   & 21  & \phantom{18}371 & $4.41\cdot10^{-12}$ & $7.83\cdot10^{-10}$ & $1.72\cdot10^{-9}$\\
		600   & 23  & \phantom{18}788 & $4.79\cdot10^{-12}$ & $1.23\cdot10^{-10}$ & $1.59\cdot10^{-9}$\\
		700   & 22  & \phantom{1}1486 & $1.54\cdot10^{-12}$ & $3.09\cdot10^{-10}$ & $8.39\cdot10^{-10}$\\
		800   & 22  & \phantom{1}2474 & $3.33\cdot10^{-12}$ & $1.72\cdot10^{ -9}$ & $1.97\cdot10^{-9}$\\
		900   & 20  & \phantom{1}4569 & $7.25\cdot10^{-12}$ & $3.57\cdot10^{-11}$ & $7.38\cdot10^{-9}$\\
		1000  & 22  & \phantom{1}8634 & $9.00\cdot10^{-13}$ & $1.84\cdot10^{-10}$ & $5.68\cdot10^{-10}$\\
		1100  & 22  &           15129 & $9.88\cdot10^{-13}$ & $4.46\cdot10^{ -9}$ & $7.75\cdot10^{-10}$\\
		\bottomrule\\
	\end{tabular}
	\caption{Solver statistics from CSDP from the solution of Example 1. Instances of the optimization problem \eqref{eq:envelope}-\eqref{eq:envelope-dual} was solved for $m=2$, $d=5$, and different values of the degree $n$. Larger problems ($n+1 \geq 1200)$ could not be solved because of memory constraints.}
	\label{tbl:ex1-highdegree-csdp}
\end{table}


\deletethis{99
	
	orthogonalizing...
	Elapsed time is 0.346807 seconds.
	solving...
	
	num. of constraints = 200
	dim. of sdp    var  = 200,   num. of sdp  blk  =  4
	dim. of free   var  = 100
	*** convert ublk to linear blk
	********************************************************************************************
	SDPT3: homogeneous self-dual path-following algorithms
	********************************************************************************************
	version  predcorr  gam  expon
	HKM      1      0.000   1
	it pstep dstep pinfeas dinfeas  gap     mean(obj)    cputime    kap   tau    theta
	--------------------------------------------------------------------------------------------
	0|0.000|0.000|6.6e+02|1.4e+02|3.3e+06| 6.877204e-02| 0:0:00|3.3e+06|1.0e+00|1.0e+00| chol 1  1 
	1|0.003|0.003|6.6e+02|1.4e+02|3.3e+06| 1.275171e-01| 0:0:01|3.3e+06|1.0e+00|1.0e+00| chol 1  1 
	2|0.004|0.004|6.6e+02|1.4e+02|3.3e+06| 5.818469e+00| 0:0:01|3.3e+06|1.0e+00|1.0e+00| chol 1  1 
	3|0.015|0.015|6.6e+02|1.4e+02|3.3e+06| 1.840765e+01| 0:0:01|3.3e+06|1.0e+00|1.0e+00| chol 1  1 
	4|0.054|0.054|6.5e+02|1.4e+02|3.4e+06| 1.843938e+01| 0:0:01|3.2e+06|1.0e+00|9.9e-01| chol 1  1 
	5|0.189|0.189|6.1e+02|1.3e+02|3.6e+06| 1.012021e+02| 0:0:01|3.1e+06|1.0e+00|9.3e-01| chol 1  1 
	6|0.638|0.638|4.6e+02|9.7e+01|4.0e+06| 8.812533e+02| 0:0:01|1.9e+06|7.5e-01|5.3e-01| chol 1  1 
	7|0.966|0.966|3.4e+01|7.2e+00|3.1e+05| 1.969516e+03| 0:0:01|1.7e+05|7.2e-01|3.7e-02| chol 1  1 
	8|0.944|0.944|1.8e+00|3.8e-01|1.9e+04| 1.866678e+03| 0:0:02|2.1e+03|7.5e-01|2.0e-03| chol 1  1 
	9|0.613|0.613|7.8e-01|1.6e-01|5.7e+03| 7.968058e+02| 0:0:02|9.4e+01|1.0e+00|1.2e-03| chol 1  1 
	10|0.734|0.734|2.3e-01|4.9e-02|1.3e+03| 2.571399e+02| 0:0:02|3.5e+00|1.4e+00|4.7e-04| chol 1  1 
	11|0.922|0.922|2.0e-02|7.0e-03|8.5e+01| 7.723424e+00| 0:0:02|3.0e+00|1.8e+00|5.7e-05| chol 1  1 
	12|0.932|0.932|7.1e-03|4.6e-03|3.5e+01| 6.272036e+00| 0:0:02|4.2e-01|1.9e+00|2.1e-05| chol 1  1 
	13|0.912|0.912|1.0e-03|4.2e-03|5.1e+00| 5.221085e-01| 0:0:02|1.8e-01|2.0e+00|3.0e-06| chol 1  1 
	14|1.000|1.000|1.3e-04|3.3e-03|6.3e-01| 3.714647e-01| 0:0:03|2.5e-02|2.0e+00|3.8e-07| chol 1  1 
	15|0.972|0.972|2.0e-05|2.9e-03|9.9e-02| 3.660781e-01| 0:0:03|3.8e-03|2.0e+00|6.0e-08| chol 1  1 
	16|0.913|0.913|3.1e-06|2.6e-03|1.5e-02| 3.707355e-01| 0:0:03|7.9e-04|2.0e+00|9.2e-09| chol 1  1 
	17|0.916|0.916|4.7e-07|2.3e-03|2.3e-03| 3.661229e-01| 0:0:03|1.4e-04|2.0e+00|1.4e-09| chol 1  1 
	18|0.971|0.971|2.7e-08|1.2e-03|1.3e-04| 3.389636e-01| 0:0:03|1.5e-05|2.0e+00|8.0e-11| chol 1  1 
	19|0.961|0.961|2.6e-09|1.0e-04|1.3e-05| 3.129381e-01| 0:0:03|1.2e-06|2.0e+00|7.9e-12| chol 1  1 
	20|1.000|1.000|2.4e-10|2.9e-06|1.2e-06| 3.105675e-01| 0:0:04|6.6e-08|2.0e+00|7.1e-13| chol 1  1 
	21|1.000|1.000|1.8e-11|1.4e-07|3.3e-08| 3.105019e-01| 0:0:04|5.9e-09|2.0e+00|2.0e-14| chol 1  1 
	22|1.000|1.000|7.9e-11|4.3e-09|1.3e-09| 3.104986e-01| 0:0:04|1.8e-10|2.0e+00|8.0e-16| chol 1  1 
	23|1.000|1.000|3.1e-10|1.3e-10|5.6e-11| 3.104985e-01| 0:0:04|7.1e-12|2.0e+00|7.5e-17| chol 1  1 
	24|0.993|0.993|6.3e-10|1.1e-11|5.2e-12| 3.104985e-01| 0:0:04|3.4e-13|2.0e+00|1.9e-15|
	lack of progress in infeas
	-------------------------------------------------------------------
	number of iterations   = 24
	primal objective value =  3.10498468e-01
	dual   objective value =  3.10498474e-01
	gap := trace(XZ)       = 5.58e-11
	relative gap           = 4.26e-11
	actual relative gap    = -3.85e-09
	rel. primal infeas     = 3.11e-10
	rel. dual   infeas     = 1.30e-10
	norm(X), norm(y), norm(Z) = 2.0e+02, 2.4e-01, 2.2e-01
	norm(A), norm(b), norm(C) = 1.4e+02, 1.6e+02, 3.2e-01
	Total CPU time (secs)  = 4.15  
	CPU time per iteration = 0.17  
	termination code       = -9
	DIMACS: 3.1e-10  0.0e+00  1.3e-10  0.0e+00  -3.9e-09  3.4e-11
	-------------------------------------------------------------------
	
	result = 
	
	yalmiptime: 0.584000000000001
	solvertime: 4.357
	info: 'Lack of progress (SDPT3-4)'
	problem: 5

	optval =
	
	-0.310498474386613

	degP =
	
	199
	
	orthogonalizing...
	Elapsed time is 2.636525 seconds.
	solving...
	
	num. of constraints = 400
	dim. of sdp    var  = 400,   num. of sdp  blk  =  4
	dim. of free   var  = 200
	*** convert ublk to linear blk
	********************************************************************************************
	SDPT3: homogeneous self-dual path-following algorithms
	********************************************************************************************
	version  predcorr  gam  expon
	HKM      1      0.000   1
	it pstep dstep pinfeas dinfeas  gap     mean(obj)    cputime    kap   tau    theta
	--------------------------------------------------------------------------------------------
	0|0.000|0.000|1.9e+03|2.0e+02|1.3e+07|-3.805050e-02| 0:0:02|1.3e+07|1.0e+00|1.0e+00| chol 1  1 
	1|0.002|0.002|1.9e+03|2.0e+02|1.3e+07| 1.845815e+00| 0:0:03|1.3e+07|1.0e+00|1.0e+00| chol 1  1 
	2|0.002|0.002|1.9e+03|2.0e+02|1.3e+07| 1.257020e+01| 0:0:03|1.3e+07|1.0e+00|1.0e+00| chol 1  1 
	3|0.008|0.008|1.9e+03|2.0e+02|1.3e+07| 3.102804e+01| 0:0:04|1.3e+07|1.0e+00|1.0e+00| chol 1  1 
	4|0.024|0.024|1.9e+03|2.0e+02|1.3e+07| 4.381668e+01| 0:0:05|1.3e+07|1.0e+00|1.0e+00| chol 1  1 
	5|0.082|0.082|1.8e+03|1.9e+02|1.4e+07| 8.073418e+01| 0:0:06|1.3e+07|1.0e+00|9.8e-01| chol 1  1 
	6|0.289|0.289|1.6e+03|1.7e+02|1.4e+07| 4.303011e+02| 0:0:07|1.1e+07|1.0e+00|8.6e-01| chol 1  1 
	7|0.889|0.889|4.2e+02|4.5e+01|4.4e+06| 2.752413e+03| 0:0:08|3.4e+06|8.2e-01|1.8e-01| chol 1  1 
	8|0.983|0.983|9.2e+00|9.8e-01|1.0e+05| 3.443271e+03| 0:0:09|8.5e+04|8.1e-01|4.0e-03| chol 1  1 
	9|0.838|0.838|1.1e+00|1.2e-01|1.5e+04| 2.524925e+03| 0:0:10|1.5e+03|9.3e-01|5.7e-04| chol 1  1 
	10|0.644|0.644|5.4e-01|5.9e-02|5.0e+03| 9.861459e+02| 0:0:11|7.3e+01|1.2e+00|3.5e-04| chol 1  1 
	11|0.733|0.733|1.5e-01|1.8e-02|1.1e+03| 2.706457e+02| 0:0:12|2.6e+00|1.6e+00|1.3e-04| chol 1  1 
	12|0.942|0.942|1.1e-02|4.5e-03|6.5e+01| 6.414100e+00| 0:0:12|1.7e+00|1.9e+00|1.1e-05| chol 1  1 
	13|0.965|0.965|3.3e-03|3.7e-03|2.3e+01| 3.432638e+00| 0:0:13|1.8e-01|2.0e+00|3.5e-06| chol 1  1 
	14|0.934|0.934|5.1e-04|3.5e-03|3.6e+00| 4.495421e-01| 0:0:14|6.3e-02|2.0e+00|5.4e-07| chol 1  1 
	15|0.993|0.993|7.7e-05|2.9e-03|5.4e-01| 3.622728e-01| 0:0:15|9.3e-03|2.0e+00|8.2e-08| chol 1  1 
	16|0.898|0.898|1.2e-05|2.6e-03|8.0e-02| 3.693606e-01| 0:0:16|2.2e-03|2.0e+00|1.2e-08| chol 1  1 
	17|0.994|0.994|1.7e-06|2.3e-03|1.2e-02| 3.690926e-01| 0:0:17|2.2e-04|2.0e+00|1.8e-09| chol 1  1 
	18|0.929|0.929|2.2e-07|1.3e-03|1.5e-03| 3.339967e-01| 0:0:18|4.3e-05|2.0e+00|2.4e-10| chol 1  1 
	19|0.977|0.977|2.6e-08|6.0e-04|1.8e-04| 3.124612e-01| 0:0:19|4.7e-06|2.0e+00|2.8e-11| chol 1  1 
	20|1.000|1.000|4.4e-09|2.9e-05|3.1e-05| 2.935087e-01| 0:0:20|4.5e-07|2.0e+00|4.7e-12| chol 1  1 
	21|1.000|1.000|1.0e-10|1.5e-06|7.5e-07| 2.925903e-01| 0:0:20|7.8e-08|2.0e+00|1.1e-13| chol 1  1 
	22|1.000|1.000|3.6e-11|7.3e-08|2.6e-08| 2.925442e-01| 0:0:21|2.0e-09|2.0e+00|3.9e-15| chol 1  1 
	23|1.000|1.000|1.6e-10|2.2e-09|9.1e-10| 2.925419e-01| 0:0:22|6.8e-11|2.0e+00|1.5e-16| chol 1  1 
	24|1.000|1.000|5.7e-10|6.8e-11|3.1e-11| 2.925418e-01| 0:0:23|2.4e-12|2.0e+00|5.9e-17| chol 1  1 
	25|0.994|0.994|1.4e-09|3.9e-12|5.7e-12| 2.925418e-01| 0:0:24|9.5e-14|2.0e+00|1.4e-15|
	lack of progress in infeas
	-------------------------------------------------------------------
	number of iterations   = 25
	primal objective value =  2.92541780e-01
	dual   objective value =  2.92541785e-01
	gap := trace(XZ)       = 3.14e-11
	relative gap           = 2.43e-11
	actual relative gap    = -3.43e-09
	rel. primal infeas     = 5.70e-10
	rel. dual   infeas     = 6.82e-11
	norm(X), norm(y), norm(Z) = 2.8e+02, 1.7e-01, 1.6e-01
	norm(A), norm(b), norm(C) = 2.0e+02, 2.2e+02, 2.2e-01
	Total CPU time (secs)  = 24.23  
	CPU time per iteration = 0.97  
	termination code       = -9
	DIMACS: 5.7e-10  0.0e+00  6.8e-11  0.0e+00  -3.4e-09  2.0e-11
	-------------------------------------------------------------------
	
	result = 
	
	yalmiptime: 1.857
	solvertime: 24.576
	info: 'Lack of progress (SDPT3-4)'
	problem: 5

	optval =
	
	-0.292541785277766

	degP =
	
	299
	
	orthogonalizing...
	Elapsed time is 9.141506 seconds.
	solving...
	
	num. of constraints = 600
	dim. of sdp    var  = 600,   num. of sdp  blk  =  4
	dim. of free   var  = 300
	*** convert ublk to linear blk
	********************************************************************************************
	SDPT3: homogeneous self-dual path-following algorithms
	********************************************************************************************
	version  predcorr  gam  expon
	HKM      1      0.000   1
	it pstep dstep pinfeas dinfeas  gap     mean(obj)    cputime    kap   tau    theta
	--------------------------------------------------------------------------------------------
	0|0.000|0.000|3.4e+03|3.0e+02|3.6e+07|-9.137431e-03| 0:0:07|3.6e+07|1.0e+00|1.0e+00| chol 1  1 
	1|0.001|0.001|3.4e+03|3.0e+02|3.6e+07|-2.974854e+00| 0:0:10|3.6e+07|1.0e+00|1.0e+00| chol 1  1 
	2|0.001|0.001|3.4e+03|3.0e+02|3.6e+07| 2.789981e+00| 0:0:14|3.6e+07|1.0e+00|1.0e+00| chol 1  1 
	3|0.004|0.004|3.4e+03|3.0e+02|3.6e+07| 2.020979e+01| 0:0:17|3.6e+07|1.0e+00|1.0e+00| chol 1  1 
	4|0.013|0.013|3.4e+03|3.0e+02|3.7e+07| 4.091528e+01| 0:0:20|3.6e+07|1.0e+00|1.0e+00| chol 1  1 
	5|0.043|0.043|3.4e+03|3.0e+02|3.8e+07| 6.672262e+01| 0:0:24|3.6e+07|1.0e+00|9.9e-01| chol 1  1 
	6|0.159|0.159|3.2e+03|2.8e+02|3.9e+07| 3.125409e+02| 0:0:27|3.4e+07|1.0e+00|9.4e-01| chol 1  1 
	7|0.539|0.539|2.8e+03|2.4e+02|5.0e+07| 2.232969e+03| 0:0:31|2.4e+07|7.7e-01|6.2e-01| chol 1  1 
	8|0.955|0.955|2.7e+02|2.3e+01|5.0e+06| 5.950537e+03| 0:0:34|2.7e+06|7.2e-01|5.5e-02| chol 1  1 
	9|0.979|0.979|5.5e+00|4.8e-01|1.2e+05| 6.193764e+03| 0:0:38|3.4e+04|7.2e-01|1.2e-03| chol 1  1 
	10|0.657|0.657|1.6e+00|1.4e-01|2.7e+04| 3.435293e+03| 0:0:41|1.0e+03|9.5e-01|4.4e-04| chol 1  1 
	11|0.696|0.696|6.4e-01|5.7e-02|8.8e+03| 1.538944e+03| 0:0:44|4.6e+01|1.2e+00|2.2e-04| chol 1  1 
	12|0.754|0.754|1.5e-01|1.5e-02|1.5e+03| 3.619433e+02| 0:0:48|1.4e+00|1.6e+00|7.4e-05| chol 1  1 
	13|0.951|0.951|9.3e-03|3.9e-03|7.9e+01| 9.392945e+00| 0:0:51|1.8e+00|2.0e+00|5.3e-06| chol 1  1 
	14|1.000|1.000|2.2e-03|4.0e-03|2.5e+01| 3.023645e+00| 0:0:55|9.8e-02|2.0e+00|1.3e-06| chol 1  1 
	15|1.000|1.000|3.8e-04|3.1e-03|4.1e+00| 4.438935e-01| 0:0:58|4.1e-02|2.0e+00|2.2e-07| chol 1  1 
	16|0.923|0.923|7.5e-05|2.7e-03|7.8e-01| 3.655498e-01| 0:1:01|9.5e-03|2.0e+00|4.3e-08| chol 1  1 
	17|0.994|0.994|1.0e-05|2.4e-03|1.1e-01| 3.698509e-01| 0:1:05|1.4e-03|2.0e+00|6.0e-09| chol 1  1 
	18|0.910|0.910|1.6e-06|2.1e-03|1.7e-02| 3.732357e-01| 0:1:08|2.9e-04|2.0e+00|9.2e-10| chol 1  1 
	19|0.931|0.931|2.4e-07|1.9e-03|2.5e-03| 3.663498e-01| 0:1:12|4.6e-05|2.0e+00|1.4e-10| chol 1  1 
	20|0.887|0.887|3.9e-08|1.1e-03|3.9e-04| 3.319628e-01| 0:1:15|8.9e-06|2.0e+00|2.3e-11| chol 1  1 
	21|0.932|0.932|9.6e-09|1.2e-04|1.0e-04| 2.938134e-01| 0:1:19|1.2e-06|2.0e+00|5.6e-12| chol 1  1 
	22|0.977|0.977|1.3e-09|5.0e-06|1.3e-05| 2.893358e-01| 0:1:22|1.9e-07|2.0e+00|7.1e-13| chol 1  1 
	23|0.929|0.929|1.1e-09|4.6e-07|1.7e-06| 2.891549e-01| 0:1:25|3.4e-08|2.0e+00|9.4e-14| chol 1  1 
	24|0.912|0.912|2.9e-09|4.6e-08|2.6e-07| 2.891384e-01| 0:1:29|5.6e-09|2.0e+00|1.3e-14| chol 1  1 
	25|0.895|0.895|2.8e-09|5.1e-09|3.8e-08| 2.891367e-01| 0:1:32|9.9e-10|2.0e+00|1.8e-15| chol 1  1 
	26|0.949|0.949|2.1e-09|2.7e-10|2.9e-09| 2.891365e-01| 0:1:36|1.1e-10|2.0e+00|1.2e-16| chol 1  1 
	27|0.942|0.942|1.5e-09|1.7e-11|2.2e-10| 2.891365e-01| 0:1:39|1.1e-11|2.0e+00|4.0e-18| chol 1  1 
	28|0.939|0.939|1.6e-09|1.1e-12|1.9e-11| 2.891365e-01| 0:1:43|1.0e-12|2.0e+00|1.6e-16| chol 1  1 
	29|0.311|0.311|8.5e-09|1.0e-12|1.5e-11| 2.891365e-01| 0:1:46|7.2e-13|2.0e+00|2.8e-15|
	lack of progress in infeas
	-------------------------------------------------------------------
	number of iterations   = 29
	primal objective value =  2.89136525e-01
	dual   objective value =  2.89136524e-01
	gap := trace(XZ)       = 1.86e-11
	relative gap           = 1.44e-11
	actual relative gap    = 1.01e-09
	rel. primal infeas     = 1.55e-09
	rel. dual   infeas     = 1.15e-12
	norm(X), norm(y), norm(Z) = 3.4e+02, 1.5e-01, 1.3e-01
	norm(A), norm(b), norm(C) = 3.0e+02, 2.7e+02, 1.8e-01
	Total CPU time (secs)  = 106.05  
	CPU time per iteration = 3.66  
	termination code       = -9
	DIMACS: 1.6e-09  0.0e+00  1.1e-12  0.0e+00  1.0e-09  1.2e-11
	-------------------------------------------------------------------
	
	result = 
	
	yalmiptime: 5.5370
	solvertime: 106.8000
	info: 'Lack of progress (SDPT3-4)'
	problem: 5

	optval =
	
	-0.289136523746832

	degP =
	
	399
	
	orthogonalizing...
	Elapsed time is 21.892281 seconds.
	solving...
	
	num. of constraints = 800
	dim. of sdp    var  = 800,   num. of sdp  blk  =  4
	dim. of free   var  = 400
	*** convert ublk to linear blk
	********************************************************************************************
	SDPT3: homogeneous self-dual path-following algorithms
	********************************************************************************************
	version  predcorr  gam  expon
	HKM      1      0.000   1
	it pstep dstep pinfeas dinfeas  gap     mean(obj)    cputime    kap   tau    theta
	--------------------------------------------------------------------------------------------
	0|0.000|0.000|5.3e+03|4.0e+02|7.4e+07| 8.715298e-03| 0:0:21|7.4e+07|1.0e+00|1.0e+00| chol 1  1 
	1|0.001|0.001|5.3e+03|4.0e+02|7.4e+07| 2.176318e+00| 0:0:30|7.4e+07|1.0e+00|1.0e+00| chol 1  1 
	2|0.001|0.001|5.3e+03|4.0e+02|7.4e+07| 1.643375e+01| 0:0:39|7.4e+07|1.0e+00|1.0e+00| chol 1  1 
	3|0.003|0.003|5.3e+03|4.0e+02|7.4e+07| 3.941423e+01| 0:0:48|7.4e+07|1.0e+00|1.0e+00| chol 1  1 
	4|0.009|0.009|5.3e+03|4.0e+02|7.5e+07| 6.844205e+01| 0:0:58|7.4e+07|1.0e+00|1.0e+00| chol 1  1 
	5|0.029|0.029|5.3e+03|4.0e+02|7.7e+07| 8.547169e+01| 0:1:07|7.4e+07|1.0e+00|1.0e+00| chol 1  1 
	6|0.108|0.108|5.1e+03|3.8e+02|7.9e+07| 2.975485e+02| 0:1:16|7.2e+07|1.0e+00|9.6e-01| chol 1  1 
	7|0.380|0.380|5.1e+03|3.8e+02|1.1e+08| 1.855328e+03| 0:1:26|5.8e+07|8.1e-01|7.7e-01| chol 1  1 
	8|0.907|0.907|1.0e+03|7.6e+01|2.5e+07| 7.393408e+03| 0:1:35|1.2e+07|7.2e-01|1.4e-01| chol 1  1 
	9|0.985|0.985|1.6e+01|1.2e+00|4.1e+05| 8.518403e+03| 0:1:44|1.7e+05|7.2e-01|2.1e-03| chol 1  1 
	10|0.831|0.831|2.2e+00|1.7e-01|6.1e+04| 6.535202e+03| 0:1:54|3.1e+03|8.3e-01|3.5e-04| chol 1  1 
	11|0.643|0.643|1.0e+00|7.8e-02|2.1e+04| 2.974760e+03| 0:2:03|1.5e+02|1.1e+00|2.1e-04| chol 1  1 
	12|0.709|0.709|3.5e-01|2.8e-02|5.6e+03| 1.127425e+03| 0:2:12|6.3e+00|1.4e+00|9.5e-05| chol 1  1 
	13|0.957|0.957|1.9e-02|4.2e-03|2.3e+02| 3.736161e+01| 0:2:21|2.9e+00|1.9e+00|6.8e-06| chol 1  1 
	14|0.976|0.976|2.4e-03|3.3e-03|3.2e+01| 1.884398e+00| 0:2:31|2.3e-01|2.0e+00|8.9e-07| chol 1  1 
	15|0.968|0.968|6.7e-04|3.1e-03|9.7e+00| 1.272030e+00| 0:2:40|4.5e-02|2.0e+00|2.5e-07| chol 1  1 
	16|0.907|0.907|1.2e-04|2.7e-03|1.7e+00| 4.307854e-01| 0:2:49|1.5e-02|2.0e+00|4.5e-08| chol 1  1 
	17|1.000|1.000|2.4e-05|2.4e-03|3.3e-01| 3.703160e-01| 0:2:59|2.1e-03|2.0e+00|8.9e-09| chol 1  1 
	18|0.878|0.878|3.5e-06|2.2e-03|4.8e-02| 3.815879e-01| 0:3:08|6.2e-04|2.0e+00|1.3e-09| chol 1  1 
	19|0.964|0.964|4.9e-07|1.9e-03|6.8e-03| 3.761133e-01| 0:3:18|8.0e-05|2.0e+00|1.8e-10| chol 1  1 
	20|0.970|0.970|6.2e-08|9.9e-04|8.6e-04| 3.334264e-01| 0:3:27|1.1e-05|2.0e+00|2.4e-11| chol 1  1 
	21|1.000|1.000|5.1e-09|4.8e-05|7.2e-05| 2.900637e-01| 0:3:36|1.1e-06|2.0e+00|1.9e-12| chol 1  1 
	22|1.000|1.000|4.0e-10|2.4e-06|4.3e-06| 2.879661e-01| 0:3:45|9.1e-08|2.0e+00|1.2e-13| chol 1  1 
	23|1.000|1.000|1.8e-10|1.2e-07|1.6e-07| 2.878613e-01| 0:3:55|5.5e-09|2.0e+00|4.2e-15| chol 1  1 
	24|1.000|1.000|5.7e-10|3.6e-09|7.9e-09| 2.878560e-01| 0:4:04|2.0e-10|2.0e+00|2.0e-16| chol 1  1 
	25|1.000|1.000|1.6e-09|1.1e-10|3.5e-10| 2.878558e-01| 0:4:13|9.9e-12|2.0e+00|1.3e-17| chol 1  1 
	26|1.000|1.000|2.7e-09|5.1e-12|1.7e-11| 2.878558e-01| 0:4:23|4.4e-13|2.0e+00|1.8e-17|
	lack of progress in infeas
	-------------------------------------------------------------------
	number of iterations   = 26
	primal objective value =  2.87855801e-01
	dual   objective value =  2.87855811e-01
	gap := trace(XZ)       = 3.47e-10
	relative gap           = 2.69e-10
	actual relative gap    = -6.64e-09
	rel. primal infeas     = 1.59e-09
	rel. dual   infeas     = 1.09e-10
	norm(X), norm(y), norm(Z) = 3.8e+02, 1.2e-01, 1.1e-01
	norm(A), norm(b), norm(C) = 4.0e+02, 3.1e+02, 1.6e-01
	Total CPU time (secs)  = 263.01  
	CPU time per iteration = 10.12  
	termination code       = -9
	DIMACS: 1.6e-09  0.0e+00  1.1e-10  0.0e+00  -6.6e-09  2.2e-10
	-------------------------------------------------------------------
	
	result = 
	
	yalmiptime: 14.4920
	solvertime: 264.4830
	info: 'Lack of progress (SDPT3-4)'
	problem: 5

	optval =
	
	-0.287855811435892

	499
	
	orthogonalizing...
	Elapsed time is 44.837287 seconds.
	solving...
	
	num. of constraints = 1000
	dim. of sdp    var  = 1000,   num. of sdp  blk  =  4
	dim. of free   var  = 500
	*** convert ublk to linear blk
	********************************************************************************************
	SDPT3: homogeneous self-dual path-following algorithms
	********************************************************************************************
	version  predcorr  gam  expon
	HKM      1      0.000   1
	it pstep dstep pinfeas dinfeas  gap     mean(obj)    cputime    kap   tau    theta
	--------------------------------------------------------------------------------------------
	0|0.000|0.000|7.4e+03|5.0e+02|1.3e+08| 1.165534e-02| 0:0:52|1.3e+08|1.0e+00|1.0e+00| chol 1  1 
	1|0.000|0.000|7.4e+03|5.0e+02|1.3e+08|-2.781380e+00| 0:1:11|1.3e+08|1.0e+00|1.0e+00| chol 1  1 
	2|0.001|0.001|7.4e+03|5.0e+02|1.3e+08|-1.068980e-01| 0:1:32|1.3e+08|1.0e+00|1.0e+00| chol 1  1 
	3|0.002|0.002|7.4e+03|5.0e+02|1.3e+08| 2.023613e+01| 0:1:54|1.3e+08|1.0e+00|1.0e+00| chol 1  1 
	4|0.006|0.006|7.4e+03|5.0e+02|1.3e+08| 4.822758e+01| 0:2:17|1.3e+08|1.0e+00|1.0e+00| chol 1  1 
	5|0.018|0.018|7.4e+03|5.0e+02|1.3e+08| 6.616416e+01| 0:2:39|1.3e+08|1.0e+00|1.0e+00| chol 1  1 
	6|0.063|0.063|7.3e+03|4.9e+02|1.4e+08| 1.971786e+02| 0:3:01|1.3e+08|1.0e+00|9.8e-01| chol 1  1 
	7|0.235|0.235|6.5e+03|4.4e+02|1.4e+08| 1.048401e+03| 0:3:23|1.1e+08|1.0e+00|8.8e-01| chol 1  1 
	8|0.815|0.815|2.3e+03|1.5e+02|5.8e+07| 6.515587e+03| 0:3:46|4.1e+07|8.3e-01|2.6e-01| chol 1  1 
	9|0.981|0.981|5.3e+01|3.6e+00|1.4e+06| 8.928714e+03| 0:4:08|1.4e+06|8.2e-01|5.9e-03| chol 1  1 
	10|0.957|0.957|2.0e+00|1.3e-01|6.8e+04| 8.435672e+03| 0:4:30|1.8e+04|8.5e-01|2.2e-04| chol 1  1 
	11|0.550|0.550|8.7e-01|6.0e-02|2.0e+04| 3.278801e+03| 0:4:53|7.8e+02|1.2e+00|1.4e-04| chol 1  1 
	12|0.700|0.700|3.1e-01|2.3e-02|5.8e+03| 1.250748e+03| 0:5:15|3.2e+01|1.5e+00|6.4e-05| chol 1  1 
	13|0.850|0.850|4.6e-02|5.3e-03|6.6e+02| 1.686910e+02| 0:5:38|7.2e-01|1.9e+00|1.2e-05| chol 1  1 
	14|0.936|0.936|3.3e-03|3.3e-03|4.4e+01| 3.627316e+00| 0:6:00|5.8e-01|2.0e+00|8.8e-07| chol 1  1 
	15|0.921|0.921|1.1e-03|3.2e-03|1.9e+01| 2.682467e+00| 0:6:23|8.4e-02|2.0e+00|3.0e-07| chol 1  1 
	16|0.902|0.902|1.9e-04|2.8e-03|3.3e+00| 4.766815e-01| 0:6:45|2.5e-02|2.0e+00|5.2e-08| chol 1  1 
	17|0.947|0.947|4.0e-05|2.4e-03|6.9e-01| 3.800549e-01| 0:7:08|4.4e-03|2.0e+00|1.1e-08| chol 1  1 
	18|0.979|0.979|5.4e-06|2.1e-03|9.4e-02| 3.862242e-01| 0:7:30|7.7e-04|2.0e+00|1.5e-09| chol 1  1 
	19|0.919|0.919|8.3e-07|1.9e-03|1.4e-02| 3.854901e-01| 0:7:53|1.5e-04|2.0e+00|2.2e-10| chol 1  1 
	20|0.931|0.931|1.3e-07|1.7e-03|2.2e-03| 3.769226e-01| 0:8:15|2.4e-05|2.0e+00|3.4e-11| chol 1  1 
	21|0.952|0.952|1.3e-08|1.7e-04|2.3e-04| 2.957771e-01| 0:8:38|3.2e-06|2.0e+00|3.6e-12| chol 1  1 
	22|0.949|0.949|3.3e-09|1.3e-05|5.5e-05| 2.879194e-01| 0:9:00|3.9e-07|2.0e+00|8.4e-13| chol 1  1 
	23|0.909|0.909|1.5e-09|1.3e-06|8.5e-06| 2.873472e-01| 0:9:23|8.5e-08|2.0e+00|1.3e-13| chol 1  1 
	24|0.931|0.931|2.6e-09|1.0e-07|1.0e-06| 2.872843e-01| 0:9:45|1.4e-08|2.0e+00|1.5e-14| chol 1  1 
	25|0.937|0.937|2.3e-09|6.8e-09|9.0e-08| 2.872794e-01| 0:10:06|1.8e-09|2.0e+00|1.2e-15| chol 1  1 
	26|0.945|0.945|3.2e-10|3.8e-10|7.0e-09| 2.872791e-01| 0:10:27|1.8e-10|2.0e+00|8.6e-17| chol 1  1 
	27|0.925|0.925|1.1e-09|3.0e-11|7.1e-10| 2.872791e-01| 0:10:49|2.0e-11|2.0e+00|8.8e-18| chol 1  1 
	28|0.918|0.918|1.6e-09|2.6e-12|7.8e-11| 2.872791e-01| 0:11:10|2.3e-12|2.0e+00|1.2e-17| chol 1  1 
	29|0.841|0.841|3.4e-09|4.3e-13|1.6e-11| 2.872791e-01| 0:11:32|4.4e-13|2.0e+00|1.9e-16|
	lack of progress in infeas
	-------------------------------------------------------------------
	number of iterations   = 29
	primal objective value =  2.87279083e-01
	dual   objective value =  2.87279083e-01
	gap := trace(XZ)       = 7.75e-11
	relative gap           = 6.02e-11
	actual relative gap    = 1.13e-10
	rel. primal infeas     = 1.58e-09
	rel. dual   infeas     = 2.56e-12
	norm(X), norm(y), norm(Z) = 4.3e+02, 1.1e-01, 9.9e-02
	norm(A), norm(b), norm(C) = 5.0e+02, 3.4e+02, 1.4e-01
	Total CPU time (secs)  = 692.00  
	CPU time per iteration = 23.86  
	termination code       = -9
	DIMACS: 1.6e-09  0.0e+00  2.6e-12  0.0e+00  1.1e-10  4.9e-11
	-------------------------------------------------------------------
	
	result = 
	
	yalmiptime: 30.575
	solvertime: 694.549
	info: 'Lack of progress (SDPT3-4)'
	problem: 5

	degP =
	
	599
	
	orthogonalizing...
	Elapsed time is 77.494461 seconds.
	solving...
	
	num. of constraints = 1200
	dim. of sdp    var  = 1200,   num. of sdp  blk  =  4
	dim. of free   var  = 600
	*** convert ublk to linear blk
	********************************************************************************************
	SDPT3: homogeneous self-dual path-following algorithms
	********************************************************************************************
	version  predcorr  gam  expon
	HKM      1      0.000   1
	it pstep dstep pinfeas dinfeas  gap     mean(obj)    cputime    kap   tau    theta
	--------------------------------------------------------------------------------------------
	0|0.000|0.000|9.7e+03|6.0e+02|2.0e+08| 1.553640e-02| 0:1:42|2.0e+08|1.0e+00|1.0e+00| chol 1  1 
	1|0.000|0.000|9.7e+03|6.0e+02|2.0e+08| 2.071400e+00| 0:2:20|2.0e+08|1.0e+00|1.0e+00| chol 1  1 
	2|0.000|0.000|9.7e+03|6.0e+02|2.0e+08| 9.377255e+00| 0:3:04|2.0e+08|1.0e+00|1.0e+00| chol 1  1 
	3|0.001|0.001|9.7e+03|6.0e+02|2.0e+08| 2.947799e+01| 0:3:49|2.0e+08|1.0e+00|1.0e+00| chol 1  1 
	4|0.003|0.003|9.7e+03|6.0e+02|2.1e+08| 6.834091e+01| 0:4:33|2.0e+08|1.0e+00|1.0e+00| chol 1  1 
	5|0.011|0.011|9.7e+03|6.0e+02|2.1e+08| 8.827972e+01| 0:5:16|2.0e+08|1.0e+00|1.0e+00| chol 1  1 
	6|0.037|0.037|9.7e+03|5.9e+02|2.1e+08| 1.646305e+02| 0:5:59|2.0e+08|1.0e+00|9.9e-01| chol 1  1 
	7|0.136|0.136|9.2e+03|5.6e+02|2.2e+08| 6.789694e+02| 0:6:42|1.9e+08|1.0e+00|9.4e-01| chol 1  1 
	8|0.551|0.551|6.7e+03|4.1e+02|2.1e+08| 4.958013e+03| 0:7:25|1.2e+08|8.2e-01|5.6e-01| chol 1  1 
	9|0.956|0.956|4.9e+02|3.0e+01|1.6e+07| 1.120707e+04| 0:8:08|1.1e+07|7.8e-01|4.0e-02| chol 1  1 
	10|0.984|0.984|7.8e+00|4.8e-01|2.8e+05| 1.157340e+04| 0:8:52|1.3e+05|7.9e-01|6.3e-04| chol 1  1 
	11|0.737|0.737|1.7e+00|1.0e-01|5.6e+04| 7.433419e+03| 0:9:35|3.2e+03|9.8e-01|1.7e-04| chol 1  1 
	12|0.663|0.663|7.1e-01|4.5e-02|1.9e+04| 3.209431e+03| 0:10:18|1.5e+02|1.3e+00|9.2e-05| chol 1  1 
	13|0.726|0.726|2.2e-01|1.5e-02|4.4e+03| 1.023104e+03| 0:11:01|5.3e+00|1.6e+00|3.6e-05| chol 1  1 
	14|0.934|0.934|1.5e-02|3.6e-03|2.5e+02| 5.684811e+01| 0:11:44|2.1e+00|1.9e+00|3.0e-06| chol 1  1 
	15|0.893|0.893|1.9e-03|3.0e-03|3.2e+01| 1.338771e+00| 0:12:27|3.6e-01|2.0e+00|3.9e-07| chol 1  1 
	16|0.965|0.965|4.6e-04|2.8e-03|9.5e+00| 1.325727e+00| 0:13:10|3.7e-02|2.0e+00|9.4e-08| chol 1  1 
	17|0.981|0.981|8.6e-05|2.4e-03|1.8e+00| 4.369891e-01| 0:13:53|8.5e-03|2.0e+00|1.8e-08| chol 1  1 
	18|1.000|1.000|1.7e-05|2.1e-03|3.5e-01| 3.776262e-01| 0:14:36|1.5e-03|2.0e+00|3.4e-09| chol 1  1 
	19|0.877|0.877|2.5e-06|1.9e-03|5.2e-02| 3.903338e-01| 0:15:18|4.4e-04|2.0e+00|5.1e-10| chol 1  1 
	20|0.924|0.924|3.7e-07|1.7e-03|7.7e-03| 3.845389e-01| 0:16:01|7.4e-05|2.0e+00|7.6e-11| chol 1  1 
	21|0.955|0.955|4.2e-08|1.6e-04|8.8e-04| 2.959055e-01| 0:16:44|9.4e-06|2.0e+00|8.6e-12| chol 1  1 
	22|0.962|0.962|5.9e-09|1.0e-05|1.2e-04| 2.875206e-01| 0:17:27|1.1e-06|2.0e+00|1.1e-12| chol 1  1 
	23|0.911|0.911|2.4e-09|1.1e-06|3.5e-05| 2.870167e-01| 0:18:10|1.8e-07|2.0e+00|3.3e-13| chol 1  1 
	24|0.734|0.734|1.1e-09|3.0e-07|1.4e-05| 2.869745e-01| 0:18:53|7.0e-08|2.0e+00|1.3e-13| chol 1  1 
	25|0.789|0.789|8.5e-10|6.4e-08|4.2e-06| 2.869625e-01| 0:19:36|2.4e-08|2.0e+00|3.8e-14| chol 1  1 
	26|0.889|0.889|1.1e-09|7.3e-09|6.8e-07| 2.869597e-01| 0:20:18|5.8e-09|2.0e+00|6.0e-15| chol 1  1 
	27|0.925|0.925|1.4e-09|5.8e-10|7.3e-08| 2.869594e-01| 0:21:02|9.6e-10|2.0e+00|6.0e-16| chol 1  1 
	28|0.954|0.954|1.2e-09|2.9e-11|4.7e-09| 2.869594e-01| 0:21:44|1.0e-10|2.0e+00|3.6e-17| chol 1  1 
	29|0.952|0.952|9.7e-10|1.6e-12|3.1e-10| 2.869594e-01| 0:22:27|8.7e-12|2.0e+00|7.4e-18| chol
	SMW too ill-conditioned, switch to LU factor, 4.2e+25.
	switch to LU factor lu  1  1 
	30|0.930|0.930|3.3e-09|1.5e-13|2.9e-11| 2.869594e-01| 0:23:10|8.5e-13|2.0e+00|1.7e-16|
	lack of progress in infeas
	-------------------------------------------------------------------
	number of iterations   = 30
	primal objective value =  2.86959377e-01
	dual   objective value =  2.86959377e-01
	gap := trace(XZ)       = 3.05e-10
	relative gap           = 2.37e-10
	actual relative gap    = 2.43e-10
	rel. primal infeas     = 9.67e-10
	rel. dual   infeas     = 1.55e-12
	norm(X), norm(y), norm(Z) = 4.7e+02, 9.8e-02, 9.1e-02
	norm(A), norm(b), norm(C) = 6.0e+02, 3.7e+02, 1.3e-01
	Total CPU time (secs)  = 1390.23  
	CPU time per iteration = 46.34  
	termination code       = -9
	DIMACS: 9.7e-10  0.0e+00  1.6e-12  0.0e+00  2.4e-10  1.9e-10
	-------------------------------------------------------------------
	
	result = 
	
	yalmiptime: 55.49
	solvertime: 1394.551
	info: 'Lack of progress (SDPT3-4)'
	problem: 5

	optval =
	
	-0.286959376949986

	699
	
	orthogonalizing...
	Elapsed time is 124.946689 seconds.
	solving...
	
	num. of constraints = 1400
	dim. of sdp    var  = 1400,   num. of sdp  blk  =  4
	dim. of free   var  = 700
	*** convert ublk to linear blk
	********************************************************************************************
	SDPT3: homogeneous self-dual path-following algorithms
	********************************************************************************************
	version  predcorr  gam  expon
	HKM      1      0.000   1
	it pstep dstep pinfeas dinfeas  gap     mean(obj)    cputime    kap   tau    theta
	--------------------------------------------------------------------------------------------
	0|0.000|0.000|1.2e+04|7.0e+02|3.0e+08| 8.908510e-03| 0:3:07|3.0e+08|1.0e+00|1.0e+00| chol 1  1 
	1|0.000|0.000|1.2e+04|7.0e+02|3.0e+08| 5.140605e-01| 0:4:14|3.0e+08|1.0e+00|1.0e+00| chol 1  1 
	2|0.000|0.000|1.2e+04|7.0e+02|3.0e+08| 9.854344e+00| 0:5:31|3.0e+08|1.0e+00|1.0e+00| chol 1  1 
	3|0.001|0.001|1.2e+04|7.0e+02|3.0e+08| 4.261383e+01| 0:6:49|3.0e+08|1.0e+00|1.0e+00| chol 1  1 
	4|0.005|0.005|1.2e+04|7.0e+02|3.0e+08| 7.539593e+01| 0:8:06|3.0e+08|1.0e+00|1.0e+00| chol 1  1 
	5|0.015|0.015|1.2e+04|7.0e+02|3.1e+08| 9.171444e+01| 0:9:24|3.0e+08|1.0e+00|1.0e+00| chol 1  1 
	6|0.050|0.050|1.2e+04|6.9e+02|3.1e+08| 2.422650e+02| 0:10:41|3.0e+08|1.0e+00|9.9e-01| chol 1  1 
	7|0.191|0.191|1.1e+04|6.3e+02|3.2e+08| 1.212044e+03| 0:11:58|2.7e+08|1.0e+00|9.1e-01| chol 1  1 
	8|0.727|0.727|5.3e+03|3.0e+02|1.9e+08| 8.152693e+03| 0:13:16|1.2e+08|8.3e-01|3.6e-01| chol 1  1 
	9|0.976|0.976|1.7e+02|9.9e+00|6.3e+06| 1.288082e+04| 0:14:33|6.0e+06|8.1e-01|1.1e-02| chol 1  1 
	10|0.977|0.977|3.7e+00|2.1e-01|1.6e+05| 1.275300e+04| 0:15:51|6.9e+04|8.2e-01|2.4e-04| chol 1  1 
	11|0.613|0.613|1.2e+00|6.8e-02|4.1e+04| 6.302459e+03| 0:17:09|2.3e+03|1.1e+00|1.1e-04| chol 1  1 
	12|0.676|0.676|4.7e-01|2.8e-02|1.3e+04| 2.547199e+03| 0:18:27|1.0e+02|1.4e+00|5.4e-05| chol 1  1 
	13|0.769|0.769|1.1e-01|8.2e-03|2.5e+03| 6.286342e+02| 0:19:45|3.1e+00|1.8e+00|1.7e-05| chol 1  1 
	14|0.931|0.931|8.2e-03|3.4e-03|1.6e+02| 3.485732e+01| 0:21:04|1.3e+00|2.0e+00|1.3e-06| chol 1  1 
	15|0.888|0.888|1.2e-03|3.7e-03|2.7e+01| 2.136672e-01| 0:22:23|2.3e-01|2.0e+00|2.0e-07| chol 1  1 
	16|0.976|0.976|3.0e-04|2.8e-03|7.3e+00| 1.147244e+00| 0:23:42|2.4e-02|2.0e+00|4.9e-08| chol 1  1 
	17|0.935|0.935|6.8e-05|2.4e-03|1.7e+00| 4.517245e-01| 0:25:02|6.4e-03|2.0e+00|1.1e-08| chol 1  1 
	18|1.000|1.000|1.5e-05|2.1e-03|3.7e-01| 3.891413e-01| 0:26:20|1.2e-03|2.0e+00|2.5e-09| chol 1  1 
	19|0.876|0.876|2.2e-06|1.9e-03|5.3e-02| 3.989929e-01| 0:27:40|3.8e-04|2.0e+00|3.6e-10| chol 1  1 
	20|0.923|0.923|3.3e-07|1.7e-03|7.9e-03| 3.922092e-01| 0:28:59|6.4e-05|2.0e+00|5.3e-11| chol 1  1 
	21|0.961|0.961|3.7e-08|1.5e-04|9.2e-04| 2.958229e-01| 0:30:19|7.9e-06|2.0e+00|6.1e-12| chol 1  1 
	22|0.938|0.938|5.1e-09|1.3e-05|1.3e-04| 2.875608e-01| 0:31:38|1.1e-06|2.0e+00|8.3e-13| chol 1  1 
	23|0.931|0.931|1.5e-09|1.1e-06|3.2e-05| 2.868301e-01| 0:32:56|1.6e-07|2.0e+00|2.1e-13| chol 1  1 
	24|0.792|0.792|1.1e-09|2.4e-07|9.6e-06| 2.867796e-01| 0:34:14|5.2e-08|2.0e+00|6.1e-14| chol 1  1 
	25|0.870|0.870|1.9e-09|3.3e-08|2.0e-06| 2.867678e-01| 0:35:31|1.3e-08|2.0e+00|1.3e-14| chol 1  1 
	26|0.854|0.854|2.9e-09|5.0e-09|4.2e-07| 2.867663e-01| 0:36:49|3.1e-09|2.0e+00|2.4e-15| chol 1  1 
	27|0.931|0.931|2.6e-09|3.6e-10|4.3e-08| 2.867661e-01| 0:38:07|4.9e-10|2.0e+00|2.3e-16| chol 1  1 
	28|0.950|0.950|2.3e-09|2.0e-11|2.9e-09| 2.867660e-01| 0:39:25|5.4e-11|2.0e+00|1.0e-17| chol 1  1 
	29|0.958|0.958|2.2e-09|9.5e-13|1.7e-10| 2.867660e-01| 0:40:42|4.3e-12|2.0e+00|0.0e+00| chol
	SMW too ill-conditioned, switch to LU factor, 5.4e+25.
	switch to LU factor lu  1  1 
	30|0.911|0.911|5.3e-09|1.2e-13|2.0e-11| 2.867660e-01| 0:42:00|4.9e-13|2.0e+00|1.4e-16|
	lack of progress in infeas
	-------------------------------------------------------------------
	number of iterations   = 30
	primal objective value =  2.86766043e-01
	dual   objective value =  2.86766043e-01
	gap := trace(XZ)       = 1.69e-10
	relative gap           = 1.31e-10
	actual relative gap    = -3.12e-10
	rel. primal infeas     = 2.19e-09
	rel. dual   infeas     = 9.46e-13
	norm(X), norm(y), norm(Z) = 5.0e+02, 9.3e-02, 8.4e-02
	norm(A), norm(b), norm(C) = 7.0e+02, 4.0e+02, 1.2e-01
	Total CPU time (secs)  = 2520.29  
	CPU time per iteration = 84.01  
	termination code       = -9
	DIMACS: 2.2e-09  0.0e+00  9.5e-13  0.0e+00  -3.1e-10  1.1e-10
	-------------------------------------------------------------------
	
	result = 
	
	yalmiptime: 90.7650000000003
	solvertime: 2526.997
	info: 'Lack of progress (SDPT3-4)'
	problem: 5

	optval =
	
	-0.286766043404821

	degP =
	
	799
	
	orthogonalizing...
	Elapsed time is 191.308743 seconds.
	solving...
	
	num. of constraints = 1600
	dim. of sdp    var  = 1600,   num. of sdp  blk  =  4
	dim. of free   var  = 800
	*** convert ublk to linear blk
	********************************************************************************************
	SDPT3: homogeneous self-dual path-following algorithms
	********************************************************************************************
	version  predcorr  gam  expon
	HKM      1      0.000   1
	it pstep dstep pinfeas dinfeas  gap     mean(obj)    cputime    kap   tau    theta
	--------------------------------------------------------------------------------------------
	0|0.000|0.000|1.5e+04|8.0e+02|4.2e+08| 6.581517e-04| 0:5:20|4.2e+08|1.0e+00|1.0e+00| chol 1  1 
	1|0.000|0.000|1.5e+04|8.0e+02|4.2e+08| 3.099542e+00| 0:7:15|4.2e+08|1.0e+00|1.0e+00| chol 1  1 
	2|0.000|0.000|1.5e+04|8.0e+02|4.2e+08| 1.633093e+01| 0:9:29|4.2e+08|1.0e+00|1.0e+00| chol 1  1 
	3|0.001|0.001|1.5e+04|8.0e+02|4.2e+08| 4.572594e+01| 0:11:43|4.2e+08|1.0e+00|1.0e+00| chol 1  1 
	4|0.003|0.003|1.5e+04|8.0e+02|4.2e+08| 9.447884e+01| 0:13:57|4.2e+08|1.0e+00|1.0e+00| chol 1  1 
	5|0.010|0.010|1.5e+04|8.0e+02|4.2e+08| 1.097426e+02| 0:16:11|4.2e+08|1.0e+00|1.0e+00| chol 1  1 
	6|0.035|0.035|1.5e+04|7.9e+02|4.3e+08| 2.260301e+02| 0:18:24|4.2e+08|1.0e+00|9.9e-01| chol 1  1 
	7|0.128|0.128|1.4e+04|7.6e+02|4.5e+08| 8.945551e+02| 0:20:37|4.0e+08|1.0e+00|9.5e-01| chol 1  1 
	8|0.490|0.490|1.2e+04|6.2e+02|4.9e+08| 5.833452e+03| 0:22:50|2.8e+08|8.2e-01|6.3e-01| chol 1  1 
	9|0.953|0.953|9.5e+02|5.1e+01|4.2e+07| 1.513033e+04| 0:25:04|2.8e+07|7.7e-01|4.9e-02| chol 1  1 
	10|0.986|0.986|1.3e+01|7.1e-01|6.3e+05| 1.584639e+04| 0:27:18|3.4e+05|7.7e-01|6.8e-04| chol 1  1 
	11|0.814|0.814|2.0e+00|1.1e-01|1.0e+05| 1.184171e+04| 0:29:32|6.5e+03|9.0e-01|1.2e-04| chol 1  1 
	12|0.625|0.625|9.3e-01|5.1e-02|3.4e+04| 5.184908e+03| 0:31:44|3.2e+02|1.2e+00|7.5e-05| chol 1  1 
	13|0.705|0.705|3.2e-01|1.8e-02|9.5e+03| 1.943770e+03| 0:33:57|1.3e+01|1.5e+00|3.3e-05| chol 1  1 
	14|0.923|0.923|2.4e-02|3.8e-03|5.4e+02| 1.350805e+02| 0:36:11|2.0e+00|1.9e+00|3.1e-06| chol 1  1 
	15|0.922|0.922|2.1e-03|3.0e-03|4.4e+01| 3.865984e+00| 0:38:25|3.9e-01|2.0e+00|2.7e-07| chol 1  1 
	16|0.946|0.946|6.3e-04|2.8e-03|1.7e+01| 2.419065e+00| 0:40:39|4.6e-02|2.0e+00|8.5e-08| chol 1  1 
	17|0.930|0.930|1.2e-04|2.5e-03|3.2e+00| 5.008395e-01| 0:42:52|1.3e-02|2.0e+00|1.6e-08| chol 1  1 
	18|0.995|0.995|2.2e-05|2.1e-03|6.1e-01| 3.876770e-01| 0:45:06|2.1e-03|2.0e+00|2.9e-09| chol 1  1 
	19|0.938|0.938|3.1e-06|1.9e-03|8.7e-02| 4.024786e-01| 0:47:20|4.9e-04|2.0e+00|4.2e-10| chol 1  1 
	20|0.918|0.918|4.8e-07|1.7e-03|1.3e-02| 3.987830e-01| 0:49:34|9.0e-05|2.0e+00|6.4e-11| chol 1  1 
	21|0.970|0.970|5.1e-08|1.4e-04|1.4e-03| 2.953410e-01| 0:51:48|1.1e-05|2.0e+00|6.9e-12| chol 1  1 
	22|0.964|0.964|6.2e-09|9.1e-06|1.7e-04| 2.872059e-01| 0:54:02|1.3e-06|2.0e+00|8.2e-13| chol 1  1 
	23|0.950|0.950|2.5e-09|6.6e-07|3.7e-05| 2.866748e-01| 0:56:16|1.7e-07|2.0e+00|1.7e-13| chol 1  1 
	24|0.788|0.788|2.8e-09|1.5e-07|1.2e-05| 2.866453e-01| 0:58:30|5.4e-08|2.0e+00|5.3e-14| chol 1  1 
	25|0.774|0.774|3.7e-09|3.5e-08|3.7e-06| 2.866390e-01| 1:00:45|1.8e-08|2.0e+00|1.6e-14| chol 1  1 
	26|0.819|0.819|6.3e-09|6.6e-09|1.0e-06| 2.866375e-01| 1:03:00|5.1e-09|2.0e+00|4.1e-15| chol 1  1 
	27|0.804|0.804|7.4e-09|1.4e-09|3.0e-07| 2.866372e-01| 1:05:13|1.5e-09|2.0e+00|1.1e-15| chol 1  1 
	28|0.823|0.823|1.1e-08|2.7e-10|7.8e-08| 2.866372e-01| 1:07:27|4.2e-10|2.0e+00|2.6e-16| chol 1  1 
	29|0.847|0.847|1.6e-08|4.6e-11|1.7e-08| 2.866372e-01| 1:09:40|1.1e-10|2.0e+00|4.5e-17| chol 1  1 
	30|0.912|0.912|1.9e-08|4.6e-12|2.1e-09| 2.866372e-01| 1:11:53|1.9e-11|2.0e+00|0.0e+00| chol
	SMW too ill-conditioned, switch to LU factor, 7.9e+26.
	switch to LU factor lu  1  1 
	31|0.953|0.953|1.9e-08|2.6e-13|1.4e-10| 2.866372e-01| 1:14:07|2.2e-12|2.0e+00|1.6e-17| lu  1  1 
	32|0.979|0.979|2.8e-08|2.1e-13|7.1e-12| 2.866372e-01| 1:16:23|1.3e-13|2.0e+00|5.1e-16| lu  1  1 
	33|0.525|0.525|3.0e-08|2.3e-13|5.4e-12| 2.866372e-01| 1:18:41|6.4e-14|2.0e+00|5.7e-16|
	Stop: relative gap < infeasibility
	-------------------------------------------------------------------
	number of iterations   = 33
	primal objective value =  2.86637256e-01
	dual   objective value =  2.86637195e-01
	gap := trace(XZ)       = 5.45e-12
	relative gap           = 4.23e-12
	actual relative gap    = 3.90e-08
	rel. primal infeas     = 3.04e-08
	rel. dual   infeas     = 2.29e-13
	norm(X), norm(y), norm(Z) = 5.4e+02, 8.7e-02, 7.9e-02
	norm(A), norm(b), norm(C) = 8.0e+02, 4.3e+02, 1.1e-01
	Total CPU time (secs)  = 4720.95  
	CPU time per iteration = 143.06  
	termination code       = -1
	DIMACS: 3.0e-08  0.0e+00  2.3e-13  0.0e+00  3.9e-08  3.5e-12
	-------------------------------------------------------------------
	
	result = 
	
	yalmiptime: 345.45
	solvertime: 4732.461
	info: 'Lack of progress (SDPT3-4)'
	problem: 5

	optval =
	
	-0.28663719501358

	899
	
	orthogonalizing...
	Elapsed time is 308.666969 seconds.
	solving...
	
	num. of constraints = 1800
	dim. of sdp    var  = 1800,   num. of sdp  blk  =  4
	dim. of free   var  = 900
	*** convert ublk to linear blk
	********************************************************************************************
	SDPT3: homogeneous self-dual path-following algorithms
	********************************************************************************************
	version  predcorr  gam  expon
	HKM      1      0.000   1
	it pstep dstep pinfeas dinfeas  gap     mean(obj)    cputime    kap   tau    theta
	--------------------------------------------------------------------------------------------
	0|0.000|0.000|1.8e+04|9.0e+02|5.6e+08|-1.271519e-03| 0:8:32|5.6e+08|1.0e+00|1.0e+00| chol 1  1 
	1|0.000|0.000|1.8e+04|9.0e+02|5.6e+08|-2.539656e+00| 0:11:33|5.6e+08|1.0e+00|1.0e+00| chol 1  1 
	2|0.000|0.000|1.8e+04|9.0e+02|5.6e+08|-2.577969e+00| 0:15:02|5.6e+08|1.0e+00|1.0e+00| chol 1  1 
	3|0.001|0.001|1.8e+04|9.0e+02|5.6e+08| 1.697725e+01| 0:18:29|5.6e+08|1.0e+00|1.0e+00| chol 1  1 
	4|0.002|0.002|1.8e+04|9.0e+02|5.6e+08| 4.890008e+01| 0:21:57|5.6e+08|1.0e+00|1.0e+00| chol 1  1 
	5|0.007|0.007|1.8e+04|9.0e+02|5.7e+08| 8.019022e+01| 0:25:24|5.6e+08|1.0e+00|1.0e+00| chol 1  1 
	6|0.024|0.024|1.8e+04|9.0e+02|5.8e+08| 1.591431e+02| 0:28:53|5.6e+08|1.0e+00|1.0e+00| chol 1  1 
	7|0.088|0.088|1.7e+04|8.7e+02|5.9e+08| 6.441472e+02| 0:32:20|5.5e+08|1.0e+00|9.7e-01| chol 1  1 
	8|0.340|0.340|1.6e+04|8.2e+02|7.4e+08| 4.106579e+03| 0:35:50|4.4e+08|8.5e-01|7.8e-01| chol 1  1 
	9|0.899|0.899|2.9e+03|1.5e+02|1.5e+08| 1.586584e+04| 0:39:17|8.3e+07|7.7e-01|1.3e-01| chol 1  1 
	10|0.986|0.986|4.1e+01|2.1e+00|2.1e+06| 1.809132e+04| 0:42:45|1.3e+06|7.7e-01|1.8e-03| chol 1  1 
	11|0.924|0.924|2.8e+00|1.4e-01|1.7e+05| 1.625715e+04| 0:46:14|1.7e+04|8.2e-01|1.3e-04| chol 1  1 
	12|0.586|0.586|1.3e+00|6.7e-02|5.5e+04| 7.200133e+03| 0:49:41|8.2e+02|1.1e+00|8.2e-05| chol 1  1 
	13|0.694|0.694|4.8e-01|2.5e-02|1.7e+04| 3.032118e+03| 0:53:07|3.5e+01|1.4e+00|3.9e-05| chol 1  1 
	14|0.853|0.853|6.8e-02|5.2e-03|1.7e+03| 4.332931e+02| 0:56:36|7.5e-01|1.9e+00|7.0e-06| chol 1  1 
	15|0.933|0.933|5.1e-03|3.0e-03|1.2e+02| 2.407374e+01| 1:00:02|8.3e-01|2.0e+00|5.6e-07| chol 1  1 
	16|0.930|0.930|7.0e-04|3.1e-03|2.0e+01|-5.918267e-01| 1:03:29|1.1e-01|2.0e+00|7.9e-08| chol 1  1 
	17|0.970|0.970|1.7e-04|2.5e-03|5.4e+00| 9.141704e-01| 1:06:58|1.4e-02|2.0e+00|1.9e-08| chol 1  1 
	18|0.943|0.943|4.1e-05|2.2e-03|1.3e+00| 4.273920e-01| 1:10:25|3.6e-03|2.0e+00|4.6e-09| chol 1  1 
	19|1.000|1.000|5.3e-06|1.9e-03|1.7e-01| 3.962825e-01| 1:13:52|7.2e-04|2.0e+00|5.9e-10| chol 1  1 
	20|0.913|0.913|8.2e-07|1.7e-03|2.6e-02| 4.037655e-01| 1:17:18|1.5e-04|2.0e+00|9.2e-11| chol 1  1 
	21|0.968|0.968|8.2e-08|2.1e-04|2.6e-03| 3.005080e-01| 1:20:46|1.9e-05|2.0e+00|9.1e-12| chol 1  1 
	22|0.968|0.968|1.5e-08|1.4e-05|4.7e-04| 2.874583e-01| 1:24:14|2.0e-06|2.0e+00|1.7e-12| chol 1  1 
	23|0.920|0.920|3.3e-09|1.5e-06|9.6e-05| 2.866384e-01| 1:27:42|4.0e-07|2.0e+00|3.3e-13| chol 1  1 
	24|0.865|0.865|1.7e-09|2.2e-07|2.1e-05| 2.865609e-01| 1:31:08|1.0e-07|2.0e+00|7.2e-14| chol 1  1 
	25|0.865|0.865|1.3e-09|3.0e-08|3.9e-06| 2.865503e-01| 1:34:35|2.4e-08|2.0e+00|1.3e-14| chol 1  1 
	26|0.942|0.942|1.1e-09|1.9e-09|3.4e-07| 2.865488e-01| 1:38:04|3.4e-09|2.0e+00|1.0e-15| chol 1  1 
	27|0.960|0.960|7.5e-10|9.4e-11|1.9e-08| 2.865487e-01| 1:41:30|3.2e-10|2.0e+00|5.5e-17| chol 1  1 
	28|0.969|0.969|5.6e-10|4.2e-12|9.1e-10| 2.865487e-01| 1:44:57|2.0e-11|2.0e+00|2.1e-18| chol
	SMW too ill-conditioned, switch to LU factor, 2.3e+26.
	switch to LU factor lu  1  1 
	29|0.974|0.974|5.4e-10|2.2e-13|3.9e-11| 2.865487e-01| 1:48:22|1.0e-12|2.0e+00|3.7e-18| lu  1  1 
	stop: primal infeas has deteriorated too much, 2.8e-08  1, 0, 0
	30|0.923|0.923|5.4e-10|2.2e-13|3.9e-11| 2.865487e-01| 1:51:48|1.0e-12|2.0e+00|3.7e-18|
	-------------------------------------------------------------------
	number of iterations   = 30
	primal objective value =  2.86548690e-01
	dual   objective value =  2.86548689e-01
	gap := trace(XZ)       = 3.91e-11
	relative gap           = 3.04e-11
	actual relative gap    = 2.22e-10
	rel. primal infeas     = 5.37e-10
	rel. dual   infeas     = 2.18e-13
	norm(X), norm(y), norm(Z) = 5.7e+02, 7.7e-02, 7.4e-02
	norm(A), norm(b), norm(C) = 9.0e+02, 4.6e+02, 1.0e-01
	Total CPU time (secs)  = 6708.36  
	CPU time per iteration = 223.61  
	termination code       = -7
	DIMACS: 5.4e-10  0.0e+00  2.2e-13  0.0e+00  2.2e-10  2.5e-11
	-------------------------------------------------------------------
	
	result = 
	
	yalmiptime: 1415.095
	solvertime: 6723.799
	info: 'Numerical problems (SDPT3-4)'
	problem: 4

	optval =
	
	-0.286548689387625

	999
	
	optval =
	
	-0.287279082882749
	
	orthogonalizing...
	Elapsed time is 871.044751 seconds.
	solving...
	
	num. of constraints = 2000
	dim. of sdp    var  = 2000,   num. of sdp  blk  =  4
	dim. of free   var  = 1000
	*** convert ublk to linear blk
	********************************************************************************************
	SDPT3: homogeneous self-dual path-following algorithms
	********************************************************************************************
	version  predcorr  gam  expon
	HKM      1      0.000   1
	it pstep dstep pinfeas dinfeas  gap     mean(obj)    cputime    kap   tau    theta
	--------------------------------------------------------------------------------------------
	0|0.000|0.000|2.1e+04|1.0e+03|7.3e+08| 2.021734e-03| 0:13:07|7.3e+08|1.0e+00|1.0e+00| chol 1  1 
	1|0.000|0.000|2.1e+04|1.0e+03|7.3e+08|-2.546908e+00| 0:17:41|7.3e+08|1.0e+00|1.0e+00| chol 1  1 
	2|0.000|0.000|2.1e+04|1.0e+03|7.3e+08|-2.923984e+00| 0:22:53|7.3e+08|1.0e+00|1.0e+00| chol 1  1 
	3|0.001|0.001|2.1e+04|1.0e+03|7.3e+08| 9.703977e+00| 0:28:05|7.3e+08|1.0e+00|1.0e+00| chol 1  1 
	4|0.002|0.002|2.1e+04|1.0e+03|7.3e+08| 4.761057e+01| 0:33:19|7.3e+08|1.0e+00|1.0e+00| chol 1  1 
	5|0.006|0.006|2.1e+04|1.0e+03|7.4e+08| 8.208923e+01| 0:38:32|7.3e+08|1.0e+00|1.0e+00| chol 1  1 
	6|0.019|0.019|2.1e+04|1.0e+03|7.5e+08| 1.374200e+02| 0:43:47|7.3e+08|1.0e+00|1.0e+00| chol 1  1 
	7|0.069|0.069|2.1e+04|9.8e+02|7.7e+08| 5.282647e+02| 0:48:58|7.2e+08|1.0e+00|9.8e-01| chol 1  1 
	8|0.262|0.262|2.1e+04|9.8e+02|9.9e+08| 3.197866e+03| 0:54:11|6.3e+08|8.6e-01|8.5e-01| chol 1  1 
	9|0.823|0.823|6.5e+03|3.1e+02|3.7e+08| 1.602779e+04| 0:59:23|1.9e+08|7.6e-01|2.3e-01| chol 1  1 
	10|0.983|0.983|1.2e+02|5.8e+00|7.1e+06| 2.069108e+04| 1:04:35|4.1e+06|7.6e-01|4.4e-03| chol 1  1 
	11|0.965|0.965|4.1e+00|2.0e-01|2.8e+05| 1.995799e+04| 1:09:46|5.0e+04|7.7e-01|1.5e-04| chol 1  1 
	12|0.580|0.580|1.7e+00|8.1e-02|8.1e+04| 9.592240e+03| 1:14:58|2.0e+03|1.1e+00|8.5e-05| chol 1  1 
	13|0.686|0.686|6.4e-01|3.2e-02|2.5e+04| 4.260482e+03| 1:20:10|8.8e+01|1.4e+00|4.2e-05| chol 1  1 
	14|0.769|0.769|1.5e-01|8.7e-03|4.7e+03| 1.099063e+03| 1:25:22|2.7e+00|1.7e+00|1.3e-05| chol 1  1 
	15|0.915|0.915|1.3e-02|3.2e-03|3.4e+02| 8.578688e+01| 1:30:35|1.8e+00|2.0e+00|1.2e-06| chol 1  1 
	16|0.906|0.906|1.4e-03|2.7e-03|3.5e+01| 1.966129e+00| 1:35:48|2.9e-01|2.0e+00|1.3e-07| chol 1  1 
	17|0.957|0.957|3.5e-04|2.5e-03|1.2e+01| 1.664667e+00| 1:41:00|2.9e-02|2.0e+00|3.3e-08| chol 1  1 
	18|0.952|0.952|6.6e-05|2.2e-03|2.3e+00| 4.738773e-01| 1:46:12|7.0e-03|2.0e+00|6.3e-09| chol 1  1 
	19|1.000|1.000|1.4e-05|1.9e-03|4.8e-01| 3.916639e-01| 1:51:24|1.1e-03|2.0e+00|1.3e-09| chol 1  1 
	20|0.883|0.883|2.0e-06|1.8e-03|7.0e-02| 4.054895e-01| 1:56:37|3.4e-04|2.0e+00|1.9e-10| chol 1  1 
	21|0.927|0.927|3.1e-07|1.6e-03|1.1e-02| 3.995333e-01| 2:01:49|5.7e-05|2.0e+00|2.9e-11| chol 1  1 
	22|0.965|0.965|3.1e-08|1.3e-04|1.1e-03| 2.958004e-01| 2:07:02|7.1e-06|2.0e+00|3.0e-12| chol 1  1 
	23|0.963|0.963|5.0e-09|8.6e-06|1.7e-04| 2.870875e-01| 2:12:14|7.9e-07|2.0e+00|4.7e-13| chol 1  1 
	24|0.888|0.888|1.9e-09|1.1e-06|4.3e-05| 2.865638e-01| 2:17:26|1.7e-07|2.0e+00|1.1e-13| chol 1  1 
	25|0.836|0.836|1.0e-09|1.9e-07|1.1e-05| 2.865000e-01| 2:22:36|4.5e-08|2.0e+00|2.9e-14| chol 1  1 
	26|0.879|0.879|8.8e-10|2.5e-08|1.9e-06| 2.864889e-01| 2:27:48|1.0e-08|2.0e+00|4.8e-15| chol 1  1 
	27|0.948|0.948|8.8e-10|1.4e-09|1.5e-07| 2.864874e-01| 2:33:01|1.4e-09|2.0e+00|3.4e-16| chol 1  1 
	28|0.959|0.959|6.7e-10|6.9e-11|8.5e-09| 2.864873e-01| 2:38:13|1.3e-10|2.0e+00|1.8e-17| chol 1  1 
	29|0.967|0.967|4.6e-10|3.3e-12|4.3e-10| 2.864873e-01| 2:43:25|8.4e-12|2.0e+00|4.1e-19| chol
	SMW too ill-conditioned, switch to LU factor, 1.3e+26.
	switch to LU factor lu  1  1 
	30|0.966|0.966|3.9e-10|2.2e-13|2.4e-11| 2.864873e-01| 2:48:37|5.0e-13|2.0e+00|2.2e-18| lu  1  1 
	stop: primal infeas has deteriorated too much, 2.5e-08  1, 0, 0
	31|0.624|0.624|3.9e-10|2.2e-13|2.4e-11| 2.864873e-01| 2:53:49|5.0e-13|2.0e+00|2.2e-18|
	-------------------------------------------------------------------
	number of iterations   = 31
	primal objective value =  2.86487319e-01
	dual   objective value =  2.86487319e-01
	gap := trace(XZ)       = 2.43e-11
	relative gap           = 1.89e-11
	actual relative gap    = 2.01e-10
	rel. primal infeas     = 3.91e-10
	rel. dual   infeas     = 2.15e-13
	norm(X), norm(y), norm(Z) = 6.0e+02, 7.7e-02, 7.0e-02
	norm(A), norm(b), norm(C) = 1.0e+03, 4.8e+02, 9.9e-02
	Total CPU time (secs)  = 10428.67  
	CPU time per iteration = 336.41  
	termination code       = -7
	DIMACS: 3.9e-10  0.0e+00  2.2e-13  0.0e+00  2.0e-10  1.5e-11
	-------------------------------------------------------------------
	
	result = 
	
	yalmiptime: 6902.138
	solvertime: 10505.128
	info: 'Numerical problems (SDPT3-4)'
	problem: 4

	optval =
	
	-0.286487318689617
	
}


\deletethis{399

orthogonalizing...
Elapsed time is 22.221395 seconds.
solving...
SeDuMi 1.34 (beta) by AdvOL, 2005-2008 and Jos F. Sturm, 1998-2003.
Alg = 2: xz-corrector, theta = 0.250, beta = 0.500
eqs m = 800, order n = 803, dim = 160402, blocks = 6
nnz(A) = 32119012 + 0, nnz(ADA) = 640000, nnz(L) = 320400
it :     b*y       gap    delta  rate   t/tP*  t/tD*   feas cg cg  prec
0 :            5.92E-02 0.000
1 :   7.76E+00 5.17E-02 0.000 0.8720 0.9000 0.9000   0.28  1  1  5.3E+00
2 :   1.24E+02 2.20E-02 0.000 0.4255 0.9000 0.9000  -0.66  1  1  4.6E+00
3 :   1.76E+02 4.06E-03 0.000 0.1846 0.9000 0.9000   0.03  1  1  1.1E+00
4 :   5.49E+01 1.46E-03 0.000 0.3587 0.9000 0.9000   3.30  1  1  1.4E-01
5 :   1.13E+01 7.08E-04 0.000 0.4864 0.9000 0.9000   4.28  1  1  2.6E-02
6 :   2.12E+00 1.92E-04 0.000 0.2705 0.9000 0.9000   2.09  1  1  4.6E-03
7 :   6.76E-01 4.40E-05 0.000 0.2298 0.9000 0.9000   1.29  1  1  9.4E-04
8 :   3.26E-01 4.18E-06 0.000 0.0950 0.9900 0.9900   1.09  1  1  8.5E-05
9 :   3.01E-01 1.38E-06 0.000 0.3293 0.9000 0.9000   1.02  1  1  2.8E-05
10 :   2.91E-01 3.82E-07 0.000 0.2771 0.9000 0.9000   1.01  1  1  7.7E-06
11 :   2.89E-01 1.09E-07 0.000 0.2855 0.9000 0.9000   1.00  1  1  2.2E-06
12 :   2.88E-01 2.95E-08 0.000 0.2711 0.9000 0.9000   1.00  2  2  5.9E-07
13 :   2.88E-01 6.78E-09 0.000 0.2296 0.9000 0.9000   1.00  2  2  1.4E-07
14 :   2.88E-01 1.35E-09 0.000 0.1998 0.9000 0.9000   1.00  2  2  2.7E-08
15 :   2.88E-01 2.86E-10 0.000 0.2112 0.9000 0.9000   1.00  2  3  5.7E-09
16 :   2.88E-01 6.21E-11 0.000 0.2169 0.9000 0.9000   1.00  3  3  1.2E-09
17 :   2.88E-01 1.48E-11 0.000 0.2387 0.9000 0.9000   1.00  3  3  3.0E-10
18 :   2.88E-01 3.72E-12 0.000 0.2513 0.9000 0.9000   1.00  4  4  7.5E-11
19 :   2.88E-01 9.47E-13 0.000 0.2546 0.9000 0.9000   1.00  6  6  1.9E-11
20 :   2.88E-01 2.39E-13 0.000 0.2521 0.9000 0.9000   1.00  8  7  4.8E-12
21 :   2.88E-01 6.01E-14 0.000 0.2516 0.9000 0.9000   1.00 11 12  1.2E-12
Run into numerical problems.

iter seconds digits       c*x               b*y
21    675.4   Inf  2.8785580128e-01  2.8785580128e-01
|Ax-b| =   2.3e-10, [Ay-c]_+ =   1.7E-13, |x|=  4.2e+02, |y|=  1.2e-01

Detailed timing (sec)
Pre          IPM          Post
1.513E+00    5.449E+02    1.710E-01    
Max-norms: ||b||=34, ||c|| = 7.873605e-03,
Cholesky |add|=0, |skip| = 34, ||L.L|| = 27.3769.

result = 

yalmiptime: 7.36300000000006
solvertime: 547.358
info: 'Numerical problems (SeDuMi-1.3)'
problem: 4

optval =

-0.287855801280624	
}

\deletethis{499
	
	orthogonalizing...
	Elapsed time is 44.863753 seconds.
	solving...
	SeDuMi 1.34 (beta) by AdvOL, 2005-2008 and Jos F. Sturm, 1998-2003.
	Alg = 2: xz-corrector, theta = 0.250, beta = 0.500
	eqs m = 1000, order n = 1003, dim = 250502, blocks = 6
	nnz(A) = 62687252 + 0, nnz(ADA) = 1000000, nnz(L) = 500500
	it :     b*y       gap    delta  rate   t/tP*  t/tD*   feas cg cg  prec
	0 :            4.74E-02 0.000
	1 :   6.74E+00 4.21E-02 0.000 0.8879 0.9000 0.9000   0.28  1  1  5.3E+00
	2 :   1.17E+02 1.92E-02 0.000 0.4560 0.9000 0.9000  -0.70  1  1  4.9E+00
	3 :   2.35E+02 1.78E-03 0.000 0.0928 0.9900 0.9900  -0.19  1  1  7.3E-01
	4 :   4.36E+01 7.45E-04 0.000 0.4187 0.9000 0.9000   5.16  1  1  7.2E-02
	5 :   7.31E+00 3.28E-04 0.000 0.4403 0.9000 0.9000   4.11  1  1  1.2E-02
	6 :   1.63E+00 8.54E-05 0.000 0.2603 0.9000 0.9000   1.71  1  1  2.4E-03
	7 :   5.66E-01 1.87E-05 0.000 0.2189 0.9000 0.9000   1.22  1  1  4.8E-04
	8 :   3.13E-01 1.65E-06 0.000 0.0885 0.9900 0.9900   1.07  1  1  4.1E-05
	9 :   2.95E-01 5.14E-07 0.000 0.3108 0.9000 0.9000   1.01  1  1  1.3E-05
	10 :   2.89E-01 1.42E-07 0.000 0.2752 0.9000 0.9000   1.01  1  1  3.5E-06
	11 :   2.88E-01 4.01E-08 0.000 0.2833 0.9000 0.9000   1.00  2  2  9.9E-07
	12 :   2.87E-01 1.10E-08 0.000 0.2740 0.9000 0.9000   1.00  2  2  2.7E-07
	13 :   2.87E-01 2.43E-09 0.000 0.2210 0.9000 0.9000   1.00  2  2  6.0E-08
	14 :   2.87E-01 4.90E-10 0.000 0.2019 0.9000 0.9000   1.00  2  2  1.2E-08
	15 :   2.87E-01 9.87E-11 0.000 0.2013 0.9000 0.9000   1.00  3  3  2.4E-09
	16 :   2.87E-01 1.96E-11 0.000 0.1990 0.9000 0.9000   1.00  3  4  4.8E-10
	17 :   2.87E-01 4.27E-12 0.000 0.2172 0.9000 0.9000   1.00  4  4  1.1E-10
	18 :   2.87E-01 9.96E-13 0.000 0.2335 0.9000 0.9000   1.00  6  6  2.5E-11
	19 :   2.87E-01 2.43E-13 0.000 0.2443 0.9000 0.9000   1.00 10 10  6.0E-12
	Run into numerical problems.
	
	iter seconds digits       c*x               b*y
	19   1296.9   Inf  2.8727908622e-01  2.8727908640e-01
	|Ax-b| =   1.1e-09, [Ay-c]_+ =   9.1E-13, |x|=  4.5e+02, |y|=  1.1e-01
	
	Detailed timing (sec)
	Pre          IPM          Post
	3.073E+00    1.123E+03    2.800E-01    
	Max-norms: ||b||=34, ||c|| = 6.295746e-03,
	Cholesky |add|=0, |skip| = 1, ||L.L|| = 73.0163.
	
	result = 
	
	yalmiptime: 14.883
	solvertime: 1128.069
	info: 'Numerical problems (SeDuMi-1.3)'
	problem: 4

	optval =
	
	-0.287279086398449
}

\deletethis{599
	
orthogonalizing...
Elapsed time is 79.158208 seconds.
solving...
SeDuMi 1.34 (beta) by AdvOL, 2005-2008 and Jos F. Sturm, 1998-2003.
Alg = 2: xz-corrector, theta = 0.250, beta = 0.500
eqs m = 1200, order n = 1203, dim = 360602, blocks = 6
nnz(A) = 108269702 + 0, nnz(ADA) = 1440000, nnz(L) = 720600
it :     b*y       gap    delta  rate   t/tP*  t/tD*   feas cg cg  prec
0 :            3.95E-02 0.000
1 :   6.01E+00 3.55E-02 0.000 0.8993 0.9000 0.9000   0.28  1  1  5.3E+00
2 :   1.11E+02 1.70E-02 0.000 0.4800 0.9000 0.9000  -0.73  1  1  5.0E+00
3 :   2.43E+02 3.63E-03 0.000 0.2128 0.9000 0.9000  -0.33  1  1  1.8E+00
4 :   1.19E+02 1.56E-03 0.000 0.4312 0.9000 0.9000   1.99  1  1  4.5E-01
5 :   3.75E+01 5.39E-04 0.000 0.3447 0.9000 0.9000   3.32  1  1  5.8E-02
6 :   6.89E+00 2.37E-04 0.000 0.4392 0.9000 0.9000   3.81  1  1  1.1E-02
7 :   1.53E+00 6.11E-05 0.000 0.2581 0.9000 0.9000   1.74  1  1  2.1E-03
8 :   5.49E-01 1.35E-05 0.000 0.2216 0.9000 0.9000   1.22  1  1  4.2E-04
9 :   3.11E-01 1.21E-06 0.000 0.0895 0.9900 0.9900   1.07  1  1  3.6E-05
10 :   2.94E-01 3.50E-07 0.000 0.2891 0.9000 0.9000   1.01  1  1  1.0E-05
11 :   2.89E-01 9.16E-08 0.000 0.2616 0.9000 0.9000   1.00  2  2  2.7E-06
12 :   2.88E-01 2.71E-08 0.000 0.2956 0.9000 0.9000   1.00  2  2  8.0E-07
13 :   2.87E-01 7.63E-09 0.000 0.2817 0.9000 0.9000   1.00  2  2  2.3E-07
14 :   2.87E-01 1.69E-09 0.000 0.2213 0.9000 0.9000   1.00  2  3  5.0E-08
15 :   2.87E-01 4.05E-10 0.000 0.2397 0.9000 0.9000   1.00  3  3  1.2E-08
16 :   2.87E-01 9.03E-11 0.000 0.2232 0.9000 0.9000   1.00  3  3  2.7E-09
17 :   2.87E-01 2.23E-11 0.000 0.2467 0.9000 0.9000   1.00  3  3  6.6E-10
18 :   2.87E-01 5.97E-12 0.000 0.2680 0.9000 0.9000   1.00  6  6  1.8E-10
19 :   2.87E-01 1.65E-12 0.000 0.2756 0.9000 0.9000   1.00  6  7  4.9E-11
20 :   2.87E-01 4.56E-13 0.000 0.2773 0.9000 0.9000   1.00 12 12  1.3E-11
Run into numerical problems.

iter seconds digits       c*x               b*y
20   2751.3   Inf  2.8695938649e-01  2.8695938650e-01
|Ax-b| =   2.6e-09, [Ay-c]_+ =   2.0E-12, |x|=  5.1e+02, |y|=  9.8e-02

Detailed timing (sec)
Pre          IPM          Post
5.476E+00    2.447E+03    4.680E-01    
Max-norms: ||b||=34, ||c|| = 5.244711e-03,
Cholesky |add|=0, |skip| = 2, ||L.L|| = 7.84548.

result = 

yalmiptime: 26.6289999999999
solvertime: 2455.591
info: 'Numerical problems (SeDuMi-1.3)'
problem: 4

optval =

-0.286959386495609
}

\deletethis{699
	
orthogonalizing...
Elapsed time is 130.171215 seconds.
solving...
SeDuMi 1.34 (beta) by AdvOL, 2005-2008 and Jos F. Sturm, 1998-2003.
Alg = 2: xz-corrector, theta = 0.250, beta = 0.500
eqs m = 1400, order n = 1403, dim = 490702, blocks = 6
nnz(A) = 171867850 + 0, nnz(ADA) = 1960000, nnz(L) = 980700
it :     b*y       gap    delta  rate   t/tP*  t/tD*   feas cg cg  prec
0 :            3.38E-02 0.000
1 :   5.47E+00 3.07E-02 0.000 0.9080 0.9000 0.9000   0.28  1  1  5.4E+00
2 :   1.06E+02 1.53E-02 0.000 0.4996 0.9000 0.9000  -0.75  1  1  5.2E+00
3 :   2.53E+02 3.88E-03 0.000 0.2530 0.9000 0.9000  -0.43  1  1  2.3E+00
4 :   1.50E+02 1.57E-03 0.000 0.4056 0.9000 0.9000   1.40  1  1  6.5E-01
5 :   5.59E+01 5.30E-04 0.000 0.3369 0.9000 0.9000   2.95  1  1  9.2E-02
6 :   1.15E+01 2.36E-04 0.000 0.4447 0.9000 0.9000   4.06  1  1  1.5E-02
7 :   2.12E+00 6.50E-05 0.000 0.2754 0.9000 0.9000   2.13  1  1  2.7E-03
8 :   6.85E-01 1.54E-05 0.000 0.2373 0.9000 0.9000   1.29  1  1  5.6E-04
9 :   3.63E-01 2.97E-06 0.000 0.1924 0.9000 0.9000   1.10  1  1  1.0E-04
10 :   2.94E-01 2.67E-07 0.000 0.0898 0.9900 0.9900   1.03  1  1  9.2E-06
11 :   2.89E-01 8.05E-08 0.000 0.3022 0.9000 0.9000   1.00  2  2  2.8E-06
12 :   2.87E-01 2.37E-08 0.000 0.2940 0.9000 0.9000   1.00  2  2  8.1E-07
13 :   2.87E-01 6.99E-09 0.000 0.2952 0.9000 0.9000   1.00  2  2  2.4E-07
14 :   2.87E-01 1.81E-09 0.000 0.2582 0.9000 0.9000   1.00  3  3  6.2E-08
15 :   2.87E-01 4.33E-10 0.000 0.2397 0.9000 0.9000   1.00  3  3  1.5E-08
16 :   2.87E-01 9.69E-11 0.000 0.2241 0.9000 0.9000   1.00  3  3  3.4E-09
17 :   2.87E-01 1.96E-11 0.000 0.2022 0.9000 0.9000   1.00  4  4  7.1E-10
18 :   2.87E-01 4.40E-12 0.000 0.2242 0.9000 0.9000   1.00  5  6  1.6E-10
19 :   2.87E-01 1.04E-12 0.000 0.2366 0.9000 0.9000   1.00  7  8  4.1E-11
20 :   2.87E-01 2.49E-13 0.000 0.2393 0.9000 0.9000   1.00  9  9  1.0E-11
21 :   2.87E-01 6.19E-14 0.000 0.2486 0.9000 0.9000   1.00 12 10  2.7E-12
Run into numerical problems.

iter seconds digits       c*x               b*y
21   5318.1  11.1  2.8676604509e-01  2.8676604508e-01
|Ax-b| =   4.8e-10, [Ay-c]_+ =   3.0E-13, |x|=  5.4e+02, |y|=  9.3e-02

Detailed timing (sec)
Pre          IPM          Post
9.048E+00    4.833E+03    7.220E-01    
Max-norms: ||b||=34, ||c|| = 4.494399e-03,
Cholesky |add|=0, |skip| = 68, ||L.L|| = 47.3408.

result = 

yalmiptime: 43.5830000000005
solvertime: 4847.119
info: 'Numerical problems (SeDuMi-1.3)'
problem: 4

optval =

-0.286766045082906}

\deletethis{799	

orthogonalizing...
Elapsed time is 196.655906 seconds.
solving...
SeDuMi 1.34 (beta) by AdvOL, 2005-2008 and Jos F. Sturm, 1998-2003.
Alg = 2: xz-corrector, theta = 0.250, beta = 0.500
eqs m = 1600, order n = 1603, dim = 640802, blocks = 6
nnz(A) = 256478806 + 0, nnz(ADA) = 2560000, nnz(L) = 1280800
it :     b*y       gap    delta  rate   t/tP*  t/tD*   feas cg cg  prec
0 :            2.96E-02 0.000
1 :   5.04E+00 2.71E-02 0.000 0.9148 0.9000 0.9000   0.28  1  1  5.4E+00
2 :   1.01E+02 1.40E-02 0.000 0.5162 0.9000 0.9000  -0.76  1  1  5.3E+00
3 :   2.55E+02 4.06E-03 0.000 0.2905 0.9000 0.9000  -0.50  1  1  2.8E+00
4 :   1.75E+02 1.46E-03 0.000 0.3601 0.9000 0.9000   0.97  1  1  8.0E-01
5 :   6.98E+01 4.74E-04 0.000 0.3245 0.9000 0.9000   2.77  1  1  1.1E-01
6 :   1.43E+01 1.93E-04 0.000 0.4079 0.9000 0.9000   4.15  1  1  1.5E-02
7 :   2.41E+00 5.66E-05 0.000 0.2926 0.9000 0.9000   2.39  1  1  2.7E-03
8 :   7.48E-01 1.37E-05 0.000 0.2428 0.9000 0.9000   1.33  1  1  5.8E-04
9 :   3.80E-01 2.79E-06 0.000 0.2029 0.9000 0.9000   1.11  1  1  1.1E-04
10 :   2.95E-01 2.46E-07 0.000 0.0884 0.9900 0.9900   1.03  1  1  9.7E-06
11 :   2.89E-01 6.92E-08 0.000 0.2809 0.9000 0.9000   1.01  2  2  2.7E-06
12 :   2.87E-01 1.97E-08 0.000 0.2851 0.9000 0.9000   1.00  2  2  7.7E-07
13 :   2.87E-01 5.84E-09 0.000 0.2957 0.9000 0.9000   1.00  2  2  2.3E-07
14 :   2.87E-01 1.46E-09 0.000 0.2507 0.9000 0.9000   1.00  3  3  5.7E-08
15 :   2.87E-01 3.07E-10 0.000 0.2094 0.9000 0.9000   1.00  3  3  1.2E-08
16 :   2.87E-01 6.41E-11 0.000 0.2092 0.9000 0.9000   1.00  3  3  2.5E-09
17 :   2.87E-01 1.31E-11 0.000 0.2038 0.9000 0.9000   1.00  4  4  5.1E-10
18 :   2.87E-01 3.09E-12 0.000 0.2362 0.9000 0.9000   1.00  6  6  1.2E-10
19 :   2.87E-01 7.94E-13 0.000 0.2571 0.9000 0.9000   1.00  9  8  3.1E-11
20 :   2.87E-01 2.10E-13 0.000 0.2643 0.9000 0.9000   1.00 14 14  8.9E-12
21 :   2.87E-01 5.76E-14 0.000 0.2745 0.9000 0.9000   1.00 14 14  3.3E-12
Run into numerical problems.

iter seconds digits       c*x               b*y
21   9175.2  11.5  2.8663719705e-01  2.8663719705e-01
|Ax-b| =   7.2e-10, [Ay-c]_+ =   3.2E-13, |x|=  5.8e+02, |y|=  8.7e-02

Detailed timing (sec)
Pre          IPM          Post
2.570E+01    8.555E+03    1.222E+00    
Max-norms: ||b||=34, ||c|| = 3.931898e-03,
Cholesky |add|=0, |skip| = 99, ||L.L|| = 38.3453.

result = 

yalmiptime: 72.7530000000006
solvertime: 8670.214
info: 'Numerical problems (SeDuMi-1.3)'
problem: 4

optval =

-0.286637197046881
}

\deletethis{899
	
	orthogonalizing...
	Elapsed time is 377.578215 seconds.
	solving...
	SeDuMi 1.34 (beta) by AdvOL, 2005-2008 and Jos F. Sturm, 1998-2003.
	Alg = 2: xz-corrector, theta = 0.250, beta = 0.500
	eqs m = 1800, order n = 1803, dim = 810902, blocks = 6
	nnz(A) = 365107052 + 0, nnz(ADA) = 3240000, nnz(L) = 1620900
	it :     b*y       gap    delta  rate   t/tP*  t/tD*   feas cg cg  prec
	0 :            2.63E-02 0.000
	1 :   4.69E+00 2.42E-02 0.000 0.9203 0.9000 0.9000   0.28  1  1  5.4E+00
	2 :   9.68E+01 1.28E-02 0.000 0.5306 0.9000 0.9000  -0.78  1  1  5.3E+00
	3 :   2.51E+02 4.17E-03 0.000 0.3245 0.9000 0.9000  -0.56  1  1  3.1E+00
	4 :   1.99E+02 1.25E-03 0.000 0.2999 0.9000 0.9000   0.65  1  1  8.6E-01
	5 :   7.81E+01 4.00E-04 0.000 0.3200 0.9000 0.9000   2.81  1  1  1.2E-01
	6 :   1.57E+01 1.56E-04 0.000 0.3896 0.9000 0.9000   4.25  1  1  1.5E-02
	7 :   2.53E+00 4.70E-05 0.000 0.3015 0.9000 0.9000   2.53  1  1  2.5E-03
	8 :   7.76E-01 1.15E-05 0.000 0.2450 0.9000 0.9000   1.34  1  1  5.4E-04
	9 :   3.87E-01 2.39E-06 0.000 0.2077 0.9000 0.9000   1.11  1  1  1.1E-04
	10 :   2.95E-01 2.02E-07 0.000 0.0847 0.9900 0.9900   1.03  1  1  8.9E-06
	11 :   2.89E-01 6.08E-08 0.000 0.3004 0.9000 0.9000   1.01  2  2  2.7E-06
	12 :   2.87E-01 1.83E-08 0.000 0.3014 0.9000 0.9000   1.00  2  2  8.1E-07
	13 :   2.87E-01 5.62E-09 0.000 0.3067 0.9000 0.9000   1.00  3  3  2.5E-07
	14 :   2.87E-01 1.45E-09 0.000 0.2571 0.9000 0.9000   1.00  3  3  6.3E-08
	15 :   2.87E-01 2.92E-10 0.000 0.2018 0.9000 0.9000   1.00  4  4  1.3E-08
	16 :   2.87E-01 5.52E-11 0.000 0.1891 0.9000 0.9000   1.00  3  4  2.4E-09
	17 :   2.87E-01 1.09E-11 0.000 0.1983 0.9000 0.9000   1.00  5  5  4.8E-10
	18 :   2.87E-01 2.55E-12 0.000 0.2328 0.9000 0.9000   1.00  6  6  1.1E-10
	19 :   2.87E-01 6.64E-13 0.000 0.2607 0.9000 0.9000   1.00  9 10  2.9E-11
	20 :   2.87E-01 1.80E-13 0.000 0.2716 0.9000 0.9000   1.00 15 18  8.0E-12
	Run into numerical problems.
	
	iter seconds digits       c*x               b*y
	20  13148.3   Inf  2.8654869747e-01  2.8654869749e-01
	|Ax-b| =   1.9e-09, [Ay-c]_+ =   1.1E-12, |x|=  6.1e+02, |y|=  7.7e-02
	
	Detailed timing (sec)
	Pre          IPM          Post
	1.158E+02    1.270E+04    1.716E+00    
	Max-norms: ||b||=34, ||c|| = 3.494536e-03,
	Cholesky |add|=0, |skip| = 457, ||L.L|| = 56253.9.
	
	result = 
	
	yalmiptime: 270.833000000001
	solvertime: 12968.732
	info: 'Numerical problems (SeDuMi-1.3)'
	problem: 4

	optval =
	
	-0.286548697485662
}

\deletethis{999
	
	orthogonalizing...
	Elapsed time is 741.345990 seconds.
	solving...
	SeDuMi 1.34 (beta) by AdvOL, 2005-2008 and Jos F. Sturm, 1998-2003.
	Alg = 2: xz-corrector, theta = 0.250, beta = 0.500
	eqs m = 2000, order n = 2003, dim = 1.001e+06, blocks = 6
	nnz(A) = 500750500 + 0, nnz(ADA) = 4000000, nnz(L) = 2001000
	it :     b*y       gap    delta  rate   t/tP*  t/tD*   feas cg cg  prec
	0 :            2.37E-02 0.000
	1 :   4.41E+00 2.19E-02 0.000 0.9250 0.9000 0.9000   0.28  1  1  5.4E+00
	2 :   9.30E+01 1.19E-02 0.000 0.5431 0.9000 0.9000  -0.79  1  1  5.4E+00
	3 :   2.44E+02 4.22E-03 0.000 0.3549 0.9000 0.9000  -0.60  1  1  3.5E+00
	4 :   2.25E+02 9.66E-04 0.000 0.2291 0.9000 0.9000   0.41  1  1  8.2E-01
	5 :   7.94E+01 3.19E-04 0.000 0.3299 0.9000 0.9000   3.16  1  1  1.0E-01
	6 :   1.55E+01 1.26E-04 0.000 0.3951 0.9000 0.9000   4.35  1  1  1.3E-02
	7 :   2.52E+00 3.77E-05 0.000 0.2994 0.9000 0.9000   2.51  1  1  2.3E-03
	8 :   7.74E-01 9.24E-06 0.000 0.2451 0.9000 0.9000   1.34  1  1  4.8E-04
	9 :   3.88E-01 1.94E-06 0.000 0.2101 0.9000 0.9000   1.11  1  1  9.6E-05
	10 :   2.95E-01 1.67E-07 0.000 0.0862 0.9900 0.9900   1.03  1  1  8.2E-06
	11 :   2.89E-01 4.39E-08 0.000 0.2621 0.9000 0.9000   1.01  2  2  2.1E-06
	12 :   2.87E-01 1.20E-08 0.000 0.2729 0.9000 0.9000   1.00  2  2  5.8E-07
	13 :   2.87E-01 3.63E-09 0.000 0.3032 0.9000 0.9000   1.00  3  3  1.8E-07
	14 :   2.87E-01 9.14E-10 0.000 0.2518 0.9000 0.9000   1.00  4  4  4.4E-08
	15 :   2.86E-01 1.97E-10 0.000 0.2161 0.9000 0.9000   1.00  4  4  9.6E-09
	16 :   2.86E-01 3.76E-11 0.000 0.1906 0.9000 0.9000   1.00  4  4  1.8E-09
	17 :   2.86E-01 7.78E-12 0.000 0.2067 0.9000 0.9000   1.00  4  5  3.8E-10
	18 :   2.86E-01 1.74E-12 0.000 0.2230 0.9000 0.9000   1.00  7  8  8.5E-11
	19 :   2.86E-01 4.19E-13 0.000 0.2413 0.9000 0.9000   1.00 11 12  2.0E-11
	Run into numerical problems.
	
	iter seconds digits       c*x               b*y
	19  19903.8   Inf  2.8648734185e-01  2.8648734197e-01
	|Ax-b| =   5.0e-09, [Ay-c]_+ =   2.8E-12, |x|=  6.4e+02, |y|=  7.7e-02
	
	Detailed timing (sec)
	Pre          IPM          Post
	1.156E+02    1.936E+04    2.246E+00    
	Max-norms: ||b||=34, ||c|| = 3.144734e-03,
	Cholesky |add|=0, |skip| = 90, ||L.L|| = 100.978.
	
	result = 
	
	yalmiptime: 916.5170
	solvertime: 1.9875e+04
	info: 'Numerical problems (SeDuMi-1.3)'
	problem: 4

	optval =
	
	-0.286487341974258
}

\subsection{Recovering best one-sided approximations}

The following example combines the semidefinite optimization model of Example \ref{ex:poly-env} with high-degree polynomial approximation of nonpolynomial functions, and it also incorporates the ideas introduced in Section \ref{sec:upsampling}. For easier verification of the results, we pose a problem that can be solved in essentially closed form using \mbox{Proposition \ref{thm:BojanicDeVore}}.

\begin{example}\label{ex:best-lower-approximation}
	Consider the problem of finding the best polynomial lower approximation (of a given degree) of the function $f(t)=\exp(t^{100})$ defined over $[-1,1]$. It can be shown using \mbox{Propositions \ref{thm:prop1} and \ref{thm:prop2}} that the polynomial interpolant $p_{200}$ of $f$ on the 200 Chebyshev points has a maximum absolute error smaller than double machine precision. For numerical purposes, finding the best polynomial lower approximant $p$ (of a given degree lower than 200) to $f$ is equivalent to finding the optimal solution $p$ to the problem
	\begin{equation}
	\begin{split}
	\textrm{maximize}_p &\quad \int_{-1}^1 p(t) dt\\
	\textrm{subject to}\, &\quad p(t)\leq p_{200}(t)\;\; \forall\,t\in [-1,1].
	\end{split}
	\end{equation}
	
This problem can be translated to a semidefinite program nearly identically to the previous example, except that now $p$ has to be upsampled to the 200 interpolation points, using barycentric interpolation, as discussed in Section \ref{sec:upsampling}. The modified dual problem becomes
	\begin{equation}
	\begin{split}
	\textrm{minimize}_{\vy,\vz} &\quad \sum_{i,\ell} p_i(t_\ell)y_{i\ell}\\
	\textrm{subject to}\;\;\,   &\quad \sum_i y_{i\ell} + z_\ell=w_\ell\quad \forall\,\ell,\\
	&\quad \sum_\ell (1+t_\ell)A^{(\ell)} y_{i\ell} \succcurlyeq 0\quad \forall\,i,\\
	&\quad \sum_\ell (1-t_\ell)A^{(\ell)} y_{i\ell} \succcurlyeq 0\quad \forall\,i,\\
	&\quad B^\T z = 0,
	\end{split}\label{eq:envelope-lowdeg-dual}
	\end{equation}
where $B\in\real^{200\times n}$ is the barycentric interpolation matrix mapping the values of a degree-$(n-1)$ polynomial $p$ at the $n$ Chebyshev points to the values at the $200$ Chebyshev points used to represent $p_{200}$.
	
	As an example, we solved this problem to determine the optimal polynomial lower approximant of degree 49 (represented as an interpolant on the 50 Chebyshev points) using SeDuMi. As before, for the highest numerically possible accuracy, we set the SeDuMi accuracy goal \texttt{eps} to zero so that the solver iterates while it can make any progress. Finally, the points of contact of the roots of optimal approximant were determined numerically using the root finding algorithm for interpolants implemented in the Matlab toolbox Chebfun v.5.0.1 \cite{chebfun}. The obtained roots are shown in Table \ref{tbl:ex2results}.
	
	Proposition \ref{thm:BojanicDeVore} provides a characterization of the optimal solution as an Hermite interpolant on the roots of the 25-degree Legendre polynomial $L_{25}$, meaning the points of contact are the (known, and numerically precisely computable) roots of $L_{25}$, allowing us to check the accuracy of our calculations. \mbox{Table \ref{tbl:ex2results}} shows the values of the roots obtained from our optimization procedure and the correct values with machine precision accuracy. The largest absolute error of the roots was $6.16\cdot 10^{-7}$, the largest relative error (not defined for the root equal to zero) was $4.88\cdot 10^{-6}$. In other words, all roots were accurate up to at least 5 significant digits.
\end{example}

\begin{table}
	\centering
\begin{tabular}{ll}
	\toprule
	Computed point of contact & Exact point of contact\\
	\midrule
        $-$0.995556972963306    &    $-$0.995556969790498\\
        $-$0.976663935477085    &    $-$0.976663921459518\\
        $-$0.942974611506432    &    $-$0.942974571228974\\
        $-$0.894992079192159    &    $-$0.894991997878275\\
        $-$0.833442768114398    &    $-$0.833442628760834\\
        $-$0.75925946688898     &   $-$0.759259263037358\\
        $-$0.673566645603713    &    $-$0.673566368473468\\
        $-$0.57766328506073     &   $-$0.577662930241223\\
        $-$0.473003173358265    &    $-$0.473002731445715\\
        $-$0.361172781372549    &    $-$0.361172305809388\\
        $-$0.243867464225739    &    $-$0.243866883720988\\
        $-$0.122865277764643    &     $-$0.12286469261071\\
        $-$6.15973190950935 $\cdot 10^{-7}$ &   \phantom{$-$}0\\
        \phantom{$-$}0.122864093106243     &     \phantom{$-$}0.12286469261071\\
        \phantom{$-$}0.243866329210549     &    \phantom{$-$}0.243866883720988\\
        \phantom{$-$}0.361171807947986     &    \phantom{$-$}0.361172305809388\\
        \phantom{$-$}0.473002281718963     &    \phantom{$-$}0.473002731445715\\
        \phantom{$-$}0.577662568403695     &    \phantom{$-$}0.577662930241223\\
        \phantom{$-$}0.673566077288833     &    \phantom{$-$}0.673566368473468\\
        \phantom{$-$}0.759259051474652     &    \phantom{$-$}0.759259263037358\\
        \phantom{$-$}0.833442487648842     &    \phantom{$-$}0.833442628760834\\
        \phantom{$-$}0.894991911861248     &    \phantom{$-$}0.894991997878275\\
        \phantom{$-$}0.942974528231816     &    \phantom{$-$}0.942974571228974\\
        \phantom{$-$}0.976663905379173     &    \phantom{$-$}0.976663921459518\\
        \phantom{$-$}0.995556966216301     &    \phantom{$-$}0.995556969790498\\
        \bottomrule\\
\end{tabular}
\caption{Comparison of the numerically computed points of contact from Example \ref{ex:best-lower-approximation} and the exact values (shown with double machine precision accuracy) derived from Proposition \ref{thm:BojanicDeVore}. In spite of the high degree of the polynomials involved, all computed points of contact (computed as the roots of high-degree sum-of-squares interpolants) are accurate up to at least 5 significant digits.}\label{tbl:ex2results}
\end{table}


\subsection{Experimental design}\label{sec:design-linear-regression}


Problems discussed in this section are among the author's main motivation for studying semi-infinite linear programs with additional conic constraints.

The goal of optimal experimental design in general is to maximize the quality of statistical inference by collecting the right data, given limited resources. In the context of linear regression, the inference is based on a data model
\begin{equation}\label{eq:model}
y(t) = \sum_{i=1}^m \beta_i f_i(t) + \epsilon(t),
\end{equation}
where $f_1,\dots,f_m$ are known functions, and the random variable $\epsilon$ (of known probability distribution) represents measurement errors and other sources of variation unexplained by the model. In the experiment, the (noisy) values of $y$ are observed for a number of different values of $t$ chosen from the given \emph{design space} $\cI$, and the goal of the experiment is to infer the values of the unknown coefficients $\beta_i$, $i=1,\dots,m$. By a \emph{design} we mean a set of values $\{t_1, \dots, t_s\}$ for which the response $y(t_i)$ is to be measured, along with the number of repeated measurements $r_i$ to be taken at each $t_i$. Multiple measurements are allowed, since assuming independent measurement errors, the repeated measurements can help reduce our uncertainty of the ``true'' (noise-free) value of each $y(t_i)$. The problem of deciding how many (discrete) measurements to take at what points $t_i$ is a non-convex (combinatorial) optimization problem, which is commonly simplified to a convex problem by relaxing the integrality constraints on $r_i$ \cite{Fedorov-72,Pukelsheim-93}. In the resulting model one can normalize the vector $r$ by assuming $\sum_i r_i = 1$ (in addition to $r\geq 0$), so that $r_i$ represents not the number, but the fraction of experiments to be conducted at the point $t_i$.

This way, the experiment design is mathematically a finitely supported probability distribution $\xi$ satisfying $\xi = t_i$ with probability $r_i$, $i=1,\dots, s$. It is immediate that the feasible set (the set of probability measures supported on a finite subset of a given set $\cI\subseteq\real^n$) is convex.

Our goal with the experiment is to maximize our confidence in the estimated components of $\beta$, which can be quantified using the Fisher information \cite{Fedorov-72,Pukelsheim-93} that the measured values carry about $\beta$. Under the common assumption that $\epsilon(t)$ is normally distributed with mean zero and variance $\omega(t)$, we can estimate the parameters $\beta$ using ordinary least-squares. Using the notation $\vf(t) = (f_1(t), \dots, f_m(t))^\T$, the \emph{Fisher information matrix} of $\beta$ corresponding to the design $\xi$ is
\begin{equation}\label{eq:Fisher} \vM(\xi) = \int_{\mathcal{I}} \vf(t) \vf(t)^\T \omega(t) d \xi(t).\end{equation}
Of course, this integral simplifies to a finite sum for every design. Note that for every $\xi$, $\vM(\xi) \in \Smp$, therefore the optimization usually takes place with respect to some \emph{optimality criterion} $\Phi$ that measures the quality of the Fisher information matrix.
If $\Phi$ is an $\Smp\mapsto\real$ function, the design $\hat\xi$ is called \emph{optimal with respect to $\Phi$}, or \emph{$\Phi$-optimal} for short, if $\Phi(\vM(\hat\xi))$ is maximum. Only those criteria $\Phi$ are interesting that are compatible with the L{\"o}wner partial order, that is functions $\Phi$ satisfying $\Phi(\vA) \geq \Phi(\vB)$ whenever $\vA \succcurlyeq \vB \succcurlyeq \vzero$; see \cite[Chap.~4]{Pukelsheim-93} for a statistical interpretation of this requirement. Popular choices of $\Phi$ include $\Phi(\vM) = \det(\vM)$, $\Phi(\vM) = \lambda_{\min}(\vM)$ (smallest eigenvalue), $\Phi(\vM) = -\tr(\vM^{-1})$, and $\Phi(\vM) = (\tr(\vM^p))^{1/p}$ for $p\geq 1$.

The main result (Theorem 2) of \cite{Papp-2012} is that if the basis functions $f_i$ in \eqref{eq:model} are rational functions, then the optimal design can be computed by solving two semidefinite programs. (For completeness, the result is repeated in the Appendix.) The first semidefinite program determines a polynomial whose roots are the support points of the optimal design, while the second one is used to determine the probability masses assigned to the support points once the support points are known. The first semidefinite program involves a constraint that a polynomial be nonnegative over the design space $\cI$; see Theorem \ref{thm:main} in the Appendix. Therefore, this is an instance of \eqref{eq:SILP} that cannot be handled effectively using the commonly used cutting plane methods of semi-infinite linear programming.

Most practical problems involve basis functions that are not polynomials or rational functions, therefore we follow the approach suggested in Section \ref{sec:introduction}, and replace each $f_i$ by a close polynomial (or rational) approximant. If any of these approximants has a high degree, then the aforementioned semidefinite program determines a high-degree polynomial that must be nonnegative over $\cI$, and that must be computed with sufficient accuracy that allows its roots to be precisely computed.

\deletethis{
\begin{theorem}[\cite{Papp-2012}]\label{thm:main}
	Suppose that in the linear model \eqref{eq:model} the design space $\mathcal{I}$ is a finite union of closed and bounded intervals, the functions $f_i$ are rational functions with finite values on $\mathcal{I}$, and $\omega$ is a nonnegative rational function on $\mathcal{I}$. Let $\Phi$ be a semidefinite representable function with representation \eqref{eq:SDP-rep_def}, admissible with respect to the set of Fisher information matrices $\Mm = \conv\{ \vf(t)\vf(t)^\T\omega(t)\,|\,t\in \mathcal{I} \}$. Then the support of the $\Phi$-optimal design is a subset of the real zeros of the polynomial $\pi$ obtained by solving the following semidefinite programming problem:
	\begin{subequations}\label{eq:final}
		\begin{align}
		\mathop{\operatorname{minimize}}_{\substack{y\in\real,\; \pi\in\real^{d+1},\\ \vW_1\in\mathbb{S}^{k_1}_+,\, \dots, \,\vW_p \in \mathbb{S}^{k_p}_+}}\; & y \\
		\operatorname{subject\,to}\quad\;\;\, & \sum_{i=1}^p \langle \vW_i, \vB_i \rangle = -1, \quad \sum_{i=1}^p C_i^*(\vW_i) = 0,\label{eq:final-BC}\\
		& \pi = \Pi(y, \vW_1, \dots, \vW_p),\label{eq:PI-constr} \\
		& \pi(t) \geq 0\quad \forall\, t\in\mathcal{I},\label{eq:POP-constr}
		\end{align}
	\end{subequations}
	where $d$ is the degree of the polynomial
	\begin{equation}\label{eq:def_pi}
	t \mapsto \lcm(\den(\omega), \den(f_1^2), \dots, \den(f_m^2)) \bigg(y - \sum_{i=1}^p \langle \vW_i, A_i(\vM(\xi_t))+\vD_i\rangle\bigg),
	\end{equation}
	whose coefficient vector is denoted by $\Pi(y, \vW_1, \dots, \vW_p)$ in \eqref{eq:PI-constr} above.
\end{theorem}

Note that the operator $\Pi$ in \eqref{eq:def_pi} is affine, hence aside from \eqref{eq:POP-constr} every constraint in \eqref{eq:final} is a linear equation or linear matrix inequality. Furthermore, \eqref{eq:POP-constr} can be translated to linear matrix inequalities using \mbox{Lemma \ref{lem:pospoly-sd-rep}}. Hence, \eqref{eq:final} is indeed a semidefinite program.
}

\begin{example}\label{ex:design-linearmodel}
Consider the linear regression model \eqref{eq:model} involving a mixture of $m=3$ Gaussians $f_i=\exp(-3(x-\mu_i)^2)$ with $\mu_1 = -0.5$, $\mu_2=0$, and $\mu_3=0.5$, and suppose we are interested in finding the best E-optimal design for determining the best fit, over the design space $\cI=[-1,1]$. Since the $f_i$ are not polynomials, we will approximate them by high-degree polynomial interpolants.

E-optimality means maximizing the smallest eigenvalue $\Phi=\lambda_{\min}$ of the Fisher information matrix. Therefore, this is a semidefinite representable optimality criterion:
\[\lambda_{\min}(\vM) \geq z \quad \iff \quad \vM - \vI z \succcurlyeq \vzero.\]
Invoking Theorem \ref{thm:main} in the Appendix with $p=1$, $A_1=\operatorname{id}$, $B_1=\vI$, $C_1(\cdot)=0$, and $D_1=\vzero$ in \eqref{eq:SDP-rep_def}, we obtain that the support of the optimal design is a subset of the roots of the optimal polynomial $\pi$ determined by the solution of the optimization problem
\begin{align*}
		\mathop{\operatorname{minimize}}_{y\in\real,\; \pi\in\real^{d+1}, \vW\in\mathbb{S}^{3}_+} & y \\
		\operatorname{subject\,to}\quad\;\;\, & \tr(\vW) = -1\\
		& \pi(t) \defeq y-\vW \bullet \vM_t \geq 0 \quad \forall\, t\in [-1,1],
\end{align*}
where $\vM_t = \vf(t)\vf(t)^T$ with $\vf(t) = (f_1(t), f_2(t), f_3(t))^\T$.

Using Chebfun (or the bound of Proposition \ref{thm:prop1}), we obtain that all nonpolynomial functions involved in the optimization problem (including not only $f_i$, but also the products $f_if_j$) can be approximated within machine-precision uniform error over $[-1,1]$ by polynomial interpolants on $40$ Chebyshev points. The nonnegativity constraint is replaced by the constraint that $\pi$ is weighted-sum-of-squares with weights $1+t$ and $1-t$.

As in the previous examples, we solved the resulting semidefinite program, and obtained the optimal polynomial shown on Figure \ref{fig:design-linearmodel}. The polynomial has three roots in $[-1,1]$, these are  $\pm 0.7410$ and $0$.
\begin{figure}[htb]
	\includegraphics[width=0.5\columnwidth]{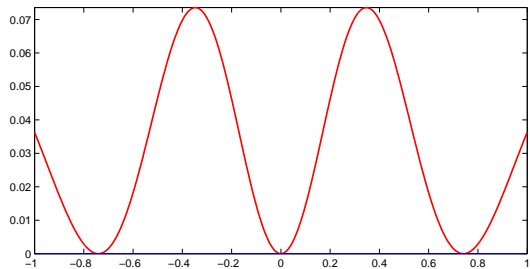}
	\caption{The optimal polynomial $\pi$ of degree $40$ from the optimal design of experiments problem discussed in Example \ref{ex:design-linearmodel}. The optimal design is supported on the roots of this polynomial located in $[-1,1]$, which are approximately $\pm 0.7410$ and $0$.}
	\label{fig:design-linearmodel}
\end{figure}


This example was also implemented in Matlab, and solved with multiple solvers. Neither solver reported any errors or warnings during the solution, and returned the same solution (within the expected accuracy).
\end{example}

\subsubsection{Extension to nonlinear regression models}
A similar approach can be used to compute optimal designs of experiments for nonlinear regression models
\[ y(t) = f(t;\beta_1,\dots,\beta_m) + \epsilon(t).\]
For such models, there are several non-equivalent definitions of optimal designs, here we only consider the simplest ones called \emph{local optimal designs}. (These are not local optimal solutions to a nonconvex optimization model; the term ``local'' has a statistical meaning in this context that will be clarified below.)

Local optimal designs for nonlinear models are defined in an almost identical manner to optimal designs for linear models, except that the in the definition \eqref{eq:Fisher} of the Fisher information matrix, the matrix $\vf(t)\vf(t)^\T$ is replaced by the matrix
\[ (\partial f(t)/\partial \beta)(\partial f(t)/\partial \beta)^\T, \]
where $\partial f(t)/\partial \beta$ denotes the vector of partial derivatives $\partial f(t)/\partial \beta_i$, $i=1,\dots,m$. It can be seen that for linear models, $\partial f(t)/\partial \beta$ is simply $\vf$, and we arrive at the previous definition. However, for nonlinear models, $\partial f(t)/\partial \beta$ is dependent on the unknown parameters $\beta$ that the experiment is designed to identify. Therefore, local optimal designs are defined with respect to an initial guess $\hat\beta$ of these parameter values; replacing the true Fisher information matrix with a local estimate,
\[ \vM_{\hat\beta}(\xi) = \int_\cI \left(\frac{\partial f(t)}{\partial \beta}\right)\left(\frac{\partial f(t)}{\partial \beta}\right)^\T\Bigg|_{\beta=\hat\beta} \omega(t) d\xi(t).\]
A design $\hat\xi$ is said to be \emph{locally optimal with respect to $\Phi$}, or \emph{locally $\Phi$-optimal} for short, if $\Phi(\vM_{\hat\beta}(\hat\xi))$ is maximum. The semidefinite programming characterization of the support of local optimal designs is entirely analogous to the characterization for linear models given in Theorem \ref{thm:main}.

\begin{example}\label{ex:design-logisticmodel}
We consider what is perhaps the simplest nonlinear regression model: logistic regression with only two parameters. In this model we have
\[ f(t) = (1+\exp(-\beta_0-\beta_1 t))^{-1},\]
and the partial derivatives of $f$ with respect to the parameters are
\[ \partial f(t)/\partial\beta_0 = (2+2\cosh(\beta_0+\beta_1 t))^{-1} \quad \text{and} \quad \partial f(t)/\partial\beta_1 = t(2+2\cosh(\beta_0+\beta_1 t))^{-1}. \]
Using the shorthand $g(t) = (2+2\cosh(\hat\beta_0+\hat\beta_1 t))^{-1}$, the optimization model characterizing the support involves as a constraint the nonnegativity of a function that is in the space
\[ \operatorname{span}\{1, g(t), t g(t), t^2 g(t)\}.\]
For our numerical example, we chose the values $\beta_0=0$ and $\beta_1=12$.

In spite of its apparent simplicity, the semidefinite program that arises from a polynomial approximation of $g$ is far from straightforward. First, it can be shown that $g$ is a function that exhibits the \emph{Runge phenomenon} \cite{Epperson-1987}, that is, the maximum (pointwise) error between $g$ and its polynomial interpolants on equispaced points does not tend to zero as the number of interpolation points increases. \mbox{(Figure \ref{fig:Runge}.)} Even if we use Chebyshev points, the 100-degree interpolant of $g$ is highly oscillatory, and yields entirely wrong results, with no correct significant digits in the optimal objective function value.

On the other hand, the 200-degree interpolant of $g$ on Chebyshev points already has a uniform error less than machine precision over $[-1,1]$. Using this approximation, we obtain that the optimal polynomial has two roots: $\pm 0.08697$. As in the previous examples, all of the SDP solvers returned the same solution without numerical problems, in spite of the high degree of the polynomials involved in the computation.
\begin{figure}[htb]
	\includegraphics[width=0.5\columnwidth]{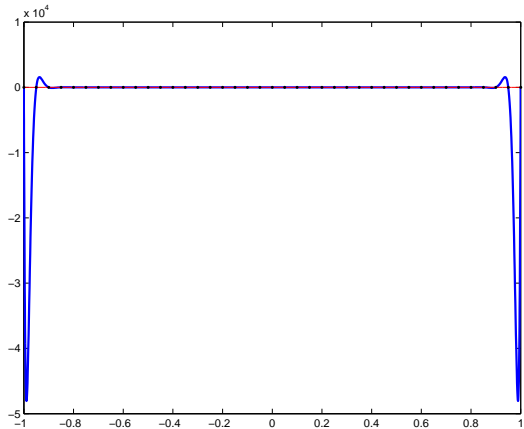}
	\caption{The function $g$ from Example \ref{ex:design-logisticmodel} exhibits the Runge phenomenon. Shown are the function $g$ in red (this bell-shaped curve appears to be a horizontal line because of the scale) and its polynomial interpolant on $40$ equispaced points.}
	\label{fig:Runge}
\end{figure}
\end{example}

\section{Discussion}\label{sec:conclusions}
While the examples presented in the previous section are admittedly toy examples of their respective application areas, the semidefinite programs they lead to are impossible to solve with commonly used semidefinite solvers if the nonnegative polynomial constraints are translated to semidefinite constraints in the standard fashion, due to the high degrees of the polynomials involved. Example \ref{ex:design-logisticmodel} in particular demonstrates that the possibly simplest problem in optimal design of experiments for nonlinear statistical models (the two-parameter version of the most studied nonlinear regression model) involves nonpolynomial functions that cannot be adequately approximated with polynomials of less than hundred degrees, therefore there is a definite need to be able to reliably optimize over cones of sum-of-squares polynomials of hundreds of degrees.

The numerical results from these experiments, on the other hand, demonstrate that the semidefinite representations of sum-of-squares interpolants enable the solution of these problems with all of the existing semidefinite programming solvers tried in this study: CSDP, SDPT3, and SeDuMi, thanks to the considerably better numerical properties of the representation. The optimal polynomials were computed with high enough accuracy that even their roots could be located precisely---this is particularly important for their application in the design of experiments.

\revision{It can be beneficial to consider rational function approximations in place of polynomial approximations.  The semidefinite programming formulations can be derived in an analogous fashion: if a basis $\{f_1,\dots,f_n\}$ of a space $F$ of functions have rational function approximations $f_i(t)\approx p_i(t)/q_i(t)$, then the cone of nonnegative functions in $F$ can be represented by the set of coefficients
\[\Big\{\alpha\in\real^n\,\Big|\,\sum_{i=1}^n \alpha_i p_i(t)/q_i(t) \geq 0\; \forall\, t\Big\} = \Big\{\alpha\in\real^n\,\Big|\,\sum_{i=1}^n \alpha_i q(t) p_i(t)/q_i(t) \geq 0\; \forall\, t\Big\},\]
where $q(t)$ is the least common multiple of the denominators $q_1,\dots,q_n$. The latter description of the cone is a preimage of the cone of nonnegative univariate polynomials, and is therefore semidefinite representable. From the perspective of this work, rational function approximations are preferred to polynomial approximations as long as the degree of the polynomials $qp_i/q_i$ are smaller than the degree required for a good polynomial approximation of the functions $f_i$.}

To derive the semidefinite programming models, it is essential that the high degree polynomial interpolants of given functions be computed quickly and accurately, and that the resulting interpolants can be processed as needed. For instance, the evaluation of interpolants at points other than the interpolation points has to be efficient and stable, and the roots of interpolants need to be computed efficiently. All the infrastructure required for these computations are already available in constructive approximation packages such as the Matlab toolbox Chebfun. Our results show that these tools can be combined with the existing semidefinite programming solvers to obtain a reliable tool that can solve semi-infinite linear programs with general univariate functional constraints, without the need to derive tailor-made specialized algorithms.

Some questions remain open, mostly around the issue of efficiency. A general criticism of sum-of-squares techniques is that the semidefinite programming representations of nonnegative polynomials of degree $n$ turn what is inherently an $O(n)$-variable optimization problem into a problem with $O(n^2)$ variables. This limits somewhat the degree of polynomials that can be practically handled in sum-of-squares optimization. The techniques presented in this paper do not address this issue, only the problem of poor scaling.

That said, the semidefinite solvers could handle all instances of the problems that could fit in the 32 GBs of memory of the desktop computer used in the experiments; this means several constraints involving 1000-degree polynomials. For applications requiring polynomials of even higher degree \revision{(or even low-degree polynomials with a large number of variables)}, it will be necessary to find either sparser representations, or tailor-made methods that can handle the higher degree sum-of-squares constraints directly, avoiding the use of semidefinite representation of size $O(n^2)$. \revision{These might also be helpful in the generalization of the approach to multivariate polynomials. The semidefinite representability of interpolants generalize relatively easily to the multivariate case, but the sizes of the semidefinite representations motivate further research into the algorithms that can solve this large-scale problems efficiently.}
\bibliographystyle{plain}
\bibliography{silp_interpolants}

\appendix

\section{A characterization of the support of optimal experimental designs}

The statement of the theorem requires two definitions and some notation.

First, elaborating on the definition used in Section \ref{sec:introduction}, we say that the concave and continuous objective function $\Phi:\Smp\mapsto\real$ of a maximization problem is \emph{semidefinite representable} if its closed upper level sets are semidefinite representable, that is, if for some $k_1, \dots, k_p$ and $\ell$ there exist linear functions $A_i:\Smp\mapsto\Ski$, $C_i:\real^\ell\mapsto\Ski$, and matrices $\vB_i \in \Ski$, $\vD_i\in\Ski$ $(i=1, \dots, p)$ such that for all $\vX\in\Smp \text{ and } z\in\real$,
$\Phi(\vX) \geq z$ holds if and only if there exists a $\vu \in \real^\ell$ satisfying
\begin{equation}\label{eq:SDP-rep_def}
A_i(\vX) + \vB_iz + C_i(\vu) + \vD_i \succcurlyeq \vzero \quad i = 1, \dots, p.
\end{equation}

Second, we say that the semidefinite representable function $\Phi:\Smp\mapsto\real$ is \emph{admissible} with respect to the set $\mathcal{X} \subseteq \Smp$ if $\Phi$ has a representation \eqref{eq:SDP-rep_def} for which there exists an $\hat{\vX} \in \mathcal{X}$ satisfying \eqref{eq:SDP-rep_def} with strict inequality for some $z$ and $\vu$. In other words, the left-hand side of each of the $p$ inequalities can be made positive definite simultaneously for at least one $\hat{\vX} \in \mathcal{X}$.

The latter is a technical condition; most interesting functions $\Phi$ are admissible with respect to every non-empty set $\mathcal{X}$, or at least with respect to every $\mathcal{X}$ that contains a non-singular matrix. Specifically, all commonly used optimality criteria, including D-, E-, and A-optimality are semidefinite representable, and they are also admissible with respect to every set of Fisher information matrices for which the criteria is well-defined.

We shall also use the common notation that the adjoint of a linear operator $C$ is denoted by $C^*$.

Now we are ready to restate the characterization of the support of optimal experimental designs.

\begin{theorem}[\cite{Papp-2012}]\label{thm:main}
Suppose that in the linear model \eqref{eq:model} the design space $\mathcal{I}$ is a finite union of closed and bounded intervals, the functions $f_i$ are rational functions with finite values on $\mathcal{I}$, and $\omega$ is a nonnegative rational function on $\mathcal{I}$. Let $\Phi$ be a semidefinite representable function with representation \eqref{eq:SDP-rep_def}, admissible with respect to the set of Fisher information matrices $\Mm = \conv\{ \vf(t)\vf(t)^\T\omega(t)\,|\,t\in \mathcal{I} \}$. Then the support of the $\Phi$-optimal design is a subset of the real zeros of the polynomial $\pi$ obtained by solving the following semidefinite programming problem:
\begin{subequations}\label{eq:final}
	\begin{align}
	\mathop{\operatorname{minimize}}_{\substack{y\in\real,\; \pi\in\real^{d+1},\\ \vW_1\in\mathbb{S}^{k_1}_+,\, \dots, \,\vW_p \in \mathbb{S}^{k_p}_+}}\; & y \\
	\operatorname{subject\,to}\quad\;\;\, & \sum_{i=1}^p \langle \vW_i, \vB_i \rangle = -1, \quad \sum_{i=1}^p C_i^*(\vW_i) = 0,\label{eq:final-BC}\\
	& \pi = \Pi(y, \vW_1, \dots, \vW_p),\label{eq:PI-constr} \\
	& \pi(t) \geq 0\quad \forall\, t\in\mathcal{I},\label{eq:POP-constr}
	\end{align}
\end{subequations}
where $d$ is the degree of the polynomial
\begin{equation}\label{eq:def_pi}
t \mapsto \lcm(\den(\omega), \den(f_1^2), \dots, \den(f_m^2)) \bigg(y - \sum_{i=1}^p \langle \vW_i, A_i(\vM(\xi_t))+\vD_i\rangle\bigg),
\end{equation}
whose coefficient vector is denoted by $\Pi(y, \vW_1, \dots, \vW_p)$ in \eqref{eq:PI-constr} above.
\end{theorem}

Note that the operator $\Pi$ in \eqref{eq:def_pi} is affine, hence aside from \eqref{eq:POP-constr} every constraint in \eqref{eq:final} is a linear equation or linear matrix inequality. Furthermore, \eqref{eq:POP-constr} can be translated to linear matrix inequalities using a sum-of-squares interpolant representation. Hence, \eqref{eq:final} is indeed a semidefinite program.
\end{document}